\documentclass[english, 11pt]{article}
\usepackage[numbers,square]{natbib}
\usepackage{color}
\usepackage{authblk}
\usepackage{subfiles}
\usepackage{amsmath}
\usepackage{amssymb}
\usepackage{algorithm, algpseudocode}
\usepackage{amsthm}
\usepackage{graphicx}
\usepackage{caption}
\usepackage{subcaption}
\usepackage{thmtools}
\usepackage{thm-restate}
\usepackage{tikz}
\usepackage[colorlinks,citecolor=blue,urlcolor=blue, hypertexnames=false]{hyperref}
\usepackage[a4paper]{geometry}
\newcommand{\supp}{\mathrm{supp}}

\newcommand{\mrd}{\mathrm{d}}
\newcommand{\Beta}{\mathrm{Beta}}
\newcommand{\FT}[1]{\widetilde{#1}}

\theoremstyle{plain}
\newtheorem{theorem}{Theorem}[section]
\newtheorem{lemma}[theorem]{Lemma}
\newtheorem{proposition}[theorem]{Proposition}
\newtheorem{corollary}[theorem]{Corollary}
\theoremstyle{definition}
\newtheorem{definition}[theorem]{Definition}

\theoremstyle{remark}
\newtheorem{remark}[theorem]{Remark}

\title{Statistical Properties of Nonparametric MLE under Laplace Noise}
\author{Yifei Xiong$^1$, Nianqiao P. Ju$^2$, Vinayak Rao$^1$}
\date{%
    $^1$Department of Statistics, Purdue University\\
    $^2$Department of Mathematics, Dartmouth College\\%
}

\begin{document} 
\maketitle

\begin{abstract}
Local differential privacy (LDP) protects individuals in a dataset by perturbing each measurement before release.
For real-valued data, a widely used mechanism is additive Laplace noise.
We study the problem of estimating the distribution of the latent confidential data from the privatized observations via the nonparametric maximum likelihood estimator (NPMLE) under an i.i.d.\ sampling model.
We first show that under the Laplace convolution model, the NPMLE admits a finite-dimensional reformulation in which the support is restricted to the observation set. 
This reduces the original infinite-dimensional optimization over all mixing distributions to an $n$-dimensional convex optimization over mixture weights, where $n$ is the sample size.
We then study the statistical convergence of the NPMLE under the 1-Wasserstein distance, and explicitly connect its convergence rate with the privacy noise scale.
Allowing the privacy noise level to change with the sample size, our analysis shows that the NPMLE remains consistent when the Laplace noise grows at a rate slower than $n^{3/16}$.
Conversely, when the Laplace noise is of the order $\sqrt n$ or larger, no estimator can achieve uniformly consistent recovery of the latent distribution.
\end{abstract}


\section{Introduction}\label{sec:intro}

\paragraph{Motivation.}Local differential privacy (LDP)~\citep{kasiviswanathan2011can,duchi2013local} extends the classical idea of randomized response~\citep{warner1965randomized} to modern data-collection settings where each user privatizes their own record before release.
It has emerged as a standard framework for collecting sensitive information when reliance on a trusted curator is either infeasible or undesirable.
Under the local model, the analyst observes only privatized reports rather than the confidential data; see, for example, \cite{cormode2018privacy,yang2024local,wang2020comprehensive}.
This client-side protection is particularly attractive in large-scale telemetry and crowdsourced data-collection settings, and LDP-based systems have been deployed in practice in a number of real-world systems~\citep{erlingsson2014rappor,ding2017collecting}.

From a statistical perspective, the key consequence of local privacy is that the analyst no longer observes i.i.d.\ samples from the population distribution of interest.
Instead, the observations are generated from a distribution obtained by applying a privatization mechanism to the latent data.
For real-valued observations, one of the most common local mechanisms is additive Laplace noise~\citep{dwork2006differential}. In the setting considered in this paper, the privatized observation from the $i$-th individual takes the form
\begin{equation*}
X_i=\theta_i+Z_i,
\end{equation*}
where $\theta_i$ is drawn from an unknown latent distribution and the privatization noise $Z_i$ is Laplace with known scale.
Consequently, the analyst does not observe samples from the latent distribution itself, but rather from its convolution with the Laplace distribution. Recovering the latent distribution from the privatized sample is therefore a deconvolution problem induced by the privacy mechanism.
As the noise level increases, this inversion becomes less stable. This is the basic tension studied in this paper: stronger privacy requires more noise, whereas more noise makes the latent distribution harder to recover.

\paragraph{Our contributions.} We study the use of the nonparametric maximum likelihood estimator (NPMLE) to solve this deconvolution problem. 
Our first contribution is computational. General mixture-likelihood theory guarantees that an NPMLE may exist with at most $n$ support points, but it does not identify where those support points are located. 
For the Laplace kernel, we prove that every support point must lie in the projected observation set. 
As a consequence, the infinite-dimensional optimization problem over all mixing distributions reduces exactly to a finite-dimensional optimization over mixture weights.
Our second contribution is statistical, seeking to assess how well we recover the latent mixing distribution. 
Because the NPMLE may be discrete even when the true latent distribution is not, we assess estimation error using the 1-Wasserstein distance.
Motivated by the Wasserstein deconvolution framework of \cite{rousseau2024wasserstein}, we adapt a deterministic inversion argument from their setting to the bounded one-dimensional Laplace privatization model considered here, while explicitly tracking the role of the noise scale.
This allows us to translate likelihood control of the privatized observation density into a 1-Wasserstein rate for the NPMLE and a sufficient condition under which consistency is preserved even when the Laplace noise increases with the sample size.
We also show a converse in the form of an impossibility result: 
if the Laplace noise scales at a rate faster than the square-root of the sample size, uniformly consistent recovery is impossible for any estimator based on the privatized sample.

\paragraph{Related work.} Our work is motivated by nonparametric density estimation under local differential privacy.
Foundational work in \cite{duchi2013local,duchi2018minimax} characterized the statistical cost of local privacy and developed minimax-optimal locally private procedures for several canonical estimation problems.
For nonparametric density estimation, \cite{19-BEJ1165} derived optimal rates under $\alpha$-differential privacy over Besov classes, focusing on density estimation over smoothness classes under locally private mechanisms.
A more closely related setting was considered by \cite{farokhi2020deconvoluting}, who developed deconvolution kernel methods for density estimation and regression from additive-noise locally private observations.
In contrast to both, we study likelihood-based estimation under the Laplace mechanism for bounded continuous data, focusing on the nonparametric maximum likelihood estimator (NPMLE) of the latent distribution and its convergence. 
Related work %
also includes inference via Monte Carlo EM methods~\citep{gong2022exact}, data-augmentation~\citep{ju2022data,xiong2025soma} and simulation-based methods~\citep{awan2025simulation,xiong2025simulation}, and frequentist confidence intervals or testing procedures for privatized or differentially private data~\citep{ferrando2022parametric,karwa2016inference,karwa2017finite}.
These works are complementary to ours, developing computational or uncertainty-quantification methods for privatized-data models. %

Our analysis of Wasserstein error is closely related to the deconvolution literature. %
For finite mixture models, \cite{nguyen2013convergence} studied the Wasserstein convergence of mixing measures.
For additive-noise deconvolution, \cite{dedecker2015improved} obtained Wasserstein rates and lower bounds in the one-dimensional ordinary-smooth setting. %
Since the Laplace density has an ordinary-smooth Fourier transform, these results provide a natural point of comparison, although their concern was %
classical deconvolution estimators rather than the likelihood-based methods.
The closest work to ours is \cite{scricciolo2018bayes}, which studies Bayesian and maximum likelihood estimation for Laplace mixtures and derives convergence of the mixing distribution in $L^1$-Wasserstein distance from convergence of the mixture density. %
More recently, \cite{rousseau2024wasserstein} developed general inversion inequalities for Bayesian and frequentist deconvolution models.
Our deconvolution argument builds on this line of work, but specializes it to the bounded-data setting and explicitly keeps track of how the noise level enters the resulting bound.

We study the nonparametric maximum likelihood estimator (NPMLE).
NPMLE is a classical approach to estimating mixing distributions, with foundational developments in \cite{laird1978nonparametric,lindsay1983geometry,lindsay1995mixture}.
It has also played an important role in empirical Bayes and compound-decision problems, including \cite{robbins1992empirical}, the convex-optimization formulation of \cite{koenker2014convex}, and recent heteroscedastic empirical Bayes models \citep{soloff2024multivariate}.
Several recent papers have established structural properties of Gaussian or Gaussian-smoothed models.
For example, \cite{polyanskiy2020self} studied self-regularization and support-size properties of the NPMLE in Gaussian mixture models, while \cite{scricciolo2026bayes} considered Bayesian nonparametric mixing-distribution estimation in a Gaussian-smoothed Wasserstein distance.
The Laplace kernel differs from the Gaussian kernel in two basic ways: it is non-differentiable at the origin and has heavier tails.
These features lead to a different likelihood geometry and motivate a separate analysis of the NPMLE support.
In this paper, we exploit the specific shape of the Laplace likelihood to show that the NPMLE support is contained in the projected observation set.
Taken together, our work combines likelihood-based estimation and deconvolution techniques to study the frequentist estimation of an unknown bounded mixing distribution from Laplace-privatized continuous observations.

\paragraph{Organization of the paper.}The rest of the paper is organized as follows. 
Section~\ref{sec:preliminaries} introduces the model setup and some preliminary results. 
Section~\ref{sec:npmle} shows that the NPMLE under the Laplace convolution model can be reduced to an optimization problem over the simplex and gives a corresponding EM algorithm. 
Section~\ref{sec:deconv} develops a deconvolution inequality for Laplace mixtures with explicit dependence on the noise scale. 
Section~\ref{sec:rate} uses these results to derive convergence rates under the 1-Wasserstein distance.
Numerical experiments are presented in Section~\ref{sec:exp}, and concluding remarks are given in Section~\ref{sec:conclusion}.

\section{Model and preliminaries}\label{sec:preliminaries}
We work in the local differential privacy (LDP) model, with $\theta_i \in \Theta$ denoting the confidential data of the $i$-th individual, and $X_i \in \mathcal X$ its locally privatized release.
We begin with the definition of $\epsilon$-local differential privacy; see, for example, \cite{duchi2013local,kairouz2016extremal}.

\begin{definition}[$\epsilon$-local differential privacy~\citep{duchi2013local}]
A randomized mechanism $\mathcal K:\Theta\to\mathcal X$ satisfies $\epsilon$-local differential privacy ($\epsilon$-LDP) if, for all $\theta,\theta'\in\Theta$ and all measurable $S\subseteq\mathcal X$,
\begin{equation}\label{eq:ldp-def}
\mathbb{P}\{\mathcal K(\theta)\in S\}\le e^\epsilon \mathbb{P}\{\mathcal K(\theta')\in S\}.
\end{equation}
When the distribution of $\mathcal K(\theta)$ has a conditional density $q(\cdot\mid\theta)$ on $\mathcal X$, a convenient sufficient condition for \eqref{eq:ldp-def} is the pointwise likelihood-ratio bound
\begin{equation*}
q(x\mid\theta)\le e^\epsilon q(x\mid\theta'),
\quad \forall x\in\mathcal X,\ \theta,\theta'\in\Theta.
\end{equation*}
\end{definition}
The parameter $\epsilon$ is referred to as the \textit{privacy loss budget}. Larger values correspond to reduced privacy guarantees, whereas $\epsilon=0$ signifies perfect privacy.
Throughout the paper, we consider the one-dimensional additive Laplace mechanism acting on a bounded latent domain, such as $\Theta = [-a,a]$ with $a>0$ fixed and known.
This boundedness assumption is commonly used in differential privacy to ensure finite sensitivity before adding noise~\citep{dwork2006differential,duchi2013local,duchi2018minimax}, arising from a range restriction or clipping step. 
Let $g_0$ denote the distribution of this bounded confidential variable and write the privatized release as
\begin{equation}\label{eq:model}
X_i=\theta_i+Z_i,\quad i=1,\dots,n,
\end{equation}
where $\theta_i \overset{iid}{\sim} g_0$, and $Z_i \overset{iid}{\sim} \mathrm{Lap}(0,b)$ has density
\begin{equation}\label{eq:laplace-density}
f_b(z)=\frac{1}{2b}\exp \left(-\frac{|z|}{b}\right),\quad z\in\mathbb R.
\end{equation}
We assume that the two sequences $\theta_i$ and $Z_i$ are independent, and the noise scale $b>0$ is known.
The next proposition gives the well-known privacy guarantee of this additive Laplace mechanism.

\begin{proposition}[Privacy guarantee for Laplace mechanism~\citep{dwork2006differential}]\label{prop:laplace-ldp}
Suppose that the latent variable $\theta$ takes values in $[-a,a]$. Then the mechanism
\begin{equation*}
\theta\mapsto \theta+Z,\quad Z\sim\mathrm{Lap}(0,b),
\end{equation*}
satisfies $\epsilon$-LDP with $\epsilon={2a}/{b}$. 
\end{proposition}

Proposition~\ref{prop:laplace-ldp} shows the basic privacy-utility tradeoff in this model. A larger $b$ yields stronger privacy, since $\epsilon=2a/b$ becomes smaller, but it also adds more noise and makes recovery of the latent distribution more difficult.

Under the observation model~\eqref{eq:model}, the privatized samples $X_1,\dots,X_n$ are i.i.d. from the mixture density
\begin{equation*}
m_{g_0}(x) := \int_{[-a,a]} f_b(x-\theta) g_0(\mathrm d\theta), \quad x\in\mathbb R.
\end{equation*}
Thus, this density is the convolution of the latent distribution $g_0$ with the Laplace kernel $f_b$.
More generally, for any probability measure $g\in\mathcal P([-a,a])$, we write the induced Laplace convolution density as 
\begin{equation*}
m_g(x) := \int_{[-a,a]} f_b(x-\theta) g(\mathrm d\theta),
\end{equation*}
where we suppress the dependence of $m_g$ on the known noise scale $b$ to simplify notation.
In our setting, the analyst only observes data from $m_{g_0}$, and we study nonparametric maximum likelihood estimation (NPMLE) of the latent mixing distribution $g_0$, defined as 
\begin{equation}\label{eq:npmle}
\widehat g_n \in \mathop{\mathrm{argmax}}\limits_{g\in\mathcal P([-a,a])} \ell_n(g), \quad
\ell_n(g):=\sum_{i=1}^n \log m_g(X_i).
\end{equation}
This is an infinite-dimensional optimization problem over the space $\mathcal P([-a,a])$ of all probability measures on $[-a,a]$.
Standard results for mixture likelihoods imply that, under mild conditions, there exists an NPMLE that is discrete with at most $n$ support points; see, for example, \cite{laird1978nonparametric,lindsay1983geometry}. 
However, this result does not identify where those support points are located. 
In the next section, we show that for the Laplace kernel, every support point of an NPMLE must lie in a projected observation set.

\section{Support reduction and computation of the NPMLE}\label{sec:npmle}

This section establishes that the support of NPMLE solution is a subset of the observations projected onto interval $[-a, a]$. 
This fact shows that the likelihood maximization over all mixing distributions reduces to a strictly convex finite-dimensional optimization problem over mixture weights.

\subsection{Support reduction}

We begin with a useful result from~\cite{polyanskiy2020self} that characterizes the support of the NPMLE.
\begin{restatable}[Support characterization~\citep{polyanskiy2020self}]{lemma}{restatesuppportpw}\label{lem:support-characterization}
Let $\widehat g_n$ be a solution to \eqref{eq:npmle}. Define
\begin{equation*}
D_{\widehat g_n}(\mu):=\frac{1}{n}\sum_{i=1}^n \frac{f_b(X_i-\mu)}{m_{\widehat g_n}(X_i)}, \quad \mu\in[-a,a].
\end{equation*}
Then $D_{\widehat g_n}(\mu)\le 1$ for all $\mu\in[-a,a]$, and $D_{\widehat g_n}(\mu)=1$ for every $\mu\in \supp(\widehat g_n)$. 
Consequently,
\begin{equation*}
\supp(\widehat g_n)\subseteq \arg\max_{\mu\in[-a,a]} D_{\widehat g_n}(\mu).
\end{equation*}
\end{restatable}
For completeness, we include a proof of this in Appendix~\ref{app:npmle}. 
In the next theorem, we show that for the Laplace mechanism, the maximizers of $D_{\widehat g_n}(\mu)$ are contained within the observations projected onto $[-a,a]$.
This implies that every support point of an NPMLE must coincide with a projected observation, thereby identifying the support locations explicitly.

\begin{theorem}[Support reduction for the Laplace NPMLE]\label{thm:laplace-support}
Let
\begin{equation*}
\Pi_{[-a,a]}(x):=\max\{-a,\min\{x,a\}\}, \quad x\in\mathbb R,
\end{equation*}
denote the projection onto $[-a,a]$. 
Then every $\widehat g_n$ solving \eqref{eq:npmle} satisfies
\begin{equation*}
\supp(\widehat g_n)\subseteq \left\{\Pi_{[-a,a]}(X_1),\dots,\Pi_{[-a,a]}(X_n)\right\}.
\end{equation*}
\end{theorem}

\begin{proof}[Proof of Theorem~\ref{thm:laplace-support}]
By Lemma~\ref{lem:support-characterization}, it suffices to show that every maximizer of
\begin{equation*}
D_{\widehat g_n}(\mu) := \frac{1}{n}\sum_{i=1}^n \frac{f_b(X_i-\mu)}{m_{\widehat g_n}(X_i)}, \quad \mu\in[-a,a],
\end{equation*}
belongs to the set $\left\{\Pi_{[-a,a]}(X_1),\dots,\Pi_{[-a,a]}(X_n)\right\}$.

Write $c_i:=\frac{1}{n m_{\widehat g_n}(X_i)}>0$ for $i=1,\dots,n$.
Since $f_b(x)=e^{-|x|/b}/2b$, we have
\begin{equation*}
D_{\widehat g_n}(\mu) = \frac{1}{2b}\sum_{i=1}^n c_i \exp\left(-\frac{|X_i-\mu|}{b}\right), \quad \mu\in[-a,a].
\end{equation*}

Let $u_1<\cdots<u_{\nu}$ be the distinct values among $\Pi_{[-a,a]}(X_1),\dots,\Pi_{[-a,a]}(X_n).$
Consider any open and connected interval $I\subset [-a,a]\setminus\{u_1,\dots,u_{\nu}\}$. 
Then $I$ contains none of the projected observations $\Pi_{[-a,a]}(X_1),\dots,\Pi_{[-a,a]}(X_n)$. Hence, for each $i$, the sign of $X_i-\mu$ is constant on $I$.
Therefore, for every $\mu\in I$, each term $\exp(-|X_i-\mu|/b)$ is twice continuously differentiable in $\mu$, and
\begin{equation*}
\frac{\mrd^2}{\mrd\mu^2}D_{\widehat g_n}(\mu) = \frac{1}{2b^3}\sum_{i=1}^n c_i \exp \left(-\frac{|X_i-\mu|}{b}\right) >0.
\end{equation*}
Thus $D_{\widehat g_n}$ is strictly convex on every connected component of $[-a,a]\setminus\{u_1,\dots,u_{\nu}\}$, and therefore cannot attain its maximum at an interior point of any such interval.
It follows that every maximizer of $D_{\widehat g_n}$ over $[-a,a]$ must belong to
\begin{equation*}
\{u_1,\dots,u_{\nu}\} = \left\{\Pi_{[-a,a]}(X_1),\dots,\Pi_{[-a,a]}(X_n)\right\}.
\end{equation*}
Applying Lemma~\ref{lem:support-characterization} yields
\begin{equation*}
\supp(\widehat g_n)\subseteq \left\{\Pi_{[-a,a]}(X_1),\dots,\Pi_{[-a,a]}(X_n)\right\}.
\end{equation*}
\end{proof}

The same argument extends to kernels of the form $\varphi(|x-\mu|)$ whenever $\varphi$ is strictly decreasing and strictly convex on $(0,\infty)$.

\begin{restatable}[Support reduction for convex radial kernels]{corollary}{restategeneralkernel}\label{cor:general-kernel-support}
Let $[a_1,a_2]\subset\mathbb R$, and suppose the density for noise $Z_i$ has the form
\begin{equation*}
f(x-\mu)=\varphi(|x-\mu|), \quad x\in\mathbb R,\ \mu\in[a_1,a_2],
\end{equation*}
where $\varphi:[0,\infty)\to(0,\infty)$ satisfies $\varphi\in C^2((0,\infty))$, and for all $t>0$, we have
$\varphi'(t)<0$ and $\varphi''(t)>0$. Let $\Pi_{[a_1,a_2]}(x):=\max\{a_1,\min\{x,a_2\}\}$, denote the projection onto $[a_1,a_2]$. Then every NPMLE $\widehat g_n$ over $\mathcal P([a_1,a_2])$ satisfies
\begin{equation*}
\supp(\widehat g_n)\subseteq \left\{\Pi_{[a_1,a_2]}(X_1),\dots,\Pi_{[a_1,a_2]}(X_n)\right\}.
\end{equation*}
\end{restatable}

\begin{remark}
The Laplace kernel is recovered by $\varphi(t)=(2b)^{-1}e^{-t/b}$.
Other examples satisfying the same shape conditions include
$\varphi(t)=C_p(1+t)^{-p}$ for $p>1$ and
$\varphi(t)=C_p e^{-t^p}$ for $0<p\le 1$,
where the constants $C_p$ normalize the kernels when needed.
\end{remark}

Theorem~\ref{thm:laplace-support} shows that, for the Laplace kernel, likelihood maximization over $\mathcal P([-a,a])$ is equivalent to maximizing over probability measures supported on the projected observation set $\{\Pi_{[-a,a]}(X_1),\dots,\Pi_{[-a,a]}(X_n)\}$.
Thus, once these support locations are fixed, the remaining optimization concerns only the mixture weights assigned to them. This leads directly to the EM algorithm in the next subsection.

\subsection{EM algorithm}

Let $u_1,\dots,u_{\nu}$ be the distinct values among $\Pi_{[-a,a]}(X_1),\ldots,\Pi_{[-a,a]}(X_n)$.
By Theorem~\ref{thm:laplace-support}, an NPMLE may be sought among distributions of the form
\begin{equation*}
g=\sum_{j=1}^{\nu} w_j\delta_{u_j}, \quad w_j\ge 0,\quad \sum_{j=1}^{\nu} w_j=1,
\end{equation*}
where $\delta_{u_j}$ denotes the Dirac probability measure at $u_j$.
Equivalently, one maximizes
\begin{equation}\label{eq:finite-loglik}
\ell_n(w_1,\dots,w_{\nu})=\sum_{i=1}^n\log \left(\sum_{j=1}^{\nu} w_j f_b(X_i-u_j)\right)
\end{equation}
over all nonnegative weights summing to one. 
The next corollary shows that this weights-only optimization problem is well-behaved.

\begin{restatable}[Strict concavity of the reduced log-likelihood]{corollary}{restatestrictconcavity}\label{cor:strict-concavity}
The function $\ell_n$ in \eqref{eq:finite-loglik} is strictly concave on the simplex
\begin{equation*}
\Delta_\nu:=\left\{w\in[0,1]^\nu:\sum_{j=1}^\nu w_j=1\right\}.
\end{equation*}
Consequently, the reduced optimization problem has a unique maximizer in the weight vector $w=(w_1,\dots,w_\nu)$.
\end{restatable}

\begin{algorithm}[t]
\caption{EM algorithm for the NPMLE under the Laplace convolution model}
\label{alg:em}
\begin{algorithmic}[1]
\Require Observations $X_1,\dots,X_n$;
tolerance $\Delta>0$
\State Compute the projected observations $\Pi_{[-a,a]}(X_1),\dots,\Pi_{[-a,a]}(X_n)$, with $u_1,\dots,u_{\nu}$ their distinct values
\State Initialize weights $w_1^{(0)},\dots,w_{\nu}^{(0)}>0$ with $\sum_{j=1}^{\nu} w_j^{(0)}=1$, set $t\gets 0$
\Repeat
\For{$i=1,\dots,n$ and $j=1,\dots,{\nu}$}
\State Compute posterior responsibility of $X_i$ for component $j$: \quad
\(
r_{ij}^{(t)}
\gets
\frac{w_j^{(t)} f_b(X_i-u_j)}
{\sum_{k=1}^{\nu} w_k^{(t)} f_b(X_i-u_k)}
\)
\EndFor
\For{$j=1,\dots,{\nu}$}
\State Update the weight for support point $u_j: \quad$
\(w_j^{(t+1)}
\gets
\frac{1}{n}\sum_{i=1}^n r_{ij}^{(t)}
\)
\EndFor
\State $t\gets t+1$
\Until{$\left|\ell_n(w^{(t)})-\ell_n(w^{(t-1)})\right|<\Delta$}
\State \Return $\widehat g_n=\sum_{j=1}^{\nu} w_j^{(t)}\delta_{u_j}$
\end{algorithmic}
\end{algorithm}

We compute this unique maximizer using the EM algorithm for a finite mixture with fixed component locations.
Given the current weights, the E-step computes the posterior responsibilities, and the M-step updates the weights by averaging these responsibilities over the sample.

Algorithm~\ref{alg:em} applies the EM updates and provides a practical way to compute the estimator.
The EM updates increase the observed log-likelihood monotonically, and standard EM convergence theory implies that every limit point is a stationary point of the reduced likelihood~\citep{dempster1977maximum,wu1983convergence}. 
Since $\ell_n$ is strictly concave on $\Delta_\nu$, any constrained stationary point is the unique global maximizer.
Thus any limit point of the EM iterates is the unique maximizer.
In our experiments, we run Algorithm~\ref{alg:em} until the likelihood increment is below the prescribed tolerance.

\section{Deconvolution inequality}\label{sec:deconv}

Section~\ref{sec:npmle} showed that the NPMLE $\widehat g_n$ under the Laplace convolution model is a discrete distribution with at most $n$ atoms, that the atoms coincide with the observations projected onto the sample space, and that it can be computed efficiently with the EM algorithm.
In this section, we develop a deconvolution inequality that will be used in Section~\ref{sec:rate} to study the convergence rate of $\widehat g_n$ as an estimator of the true mixing measure $g_0$.
Since $\widehat g_n$ is an atomic measure while $g_0$ may be arbitrary, we measure their discrepancy using the 1-Wasserstein distance~\cite{villani2009optimal}.

Recall that $g_0\in\mathcal P([-a,a])$ and let $g\in\mathcal P([-a,a])$ be any competing measure.
Write $G_0$ and $G$ for their distribution functions, respectively.
We also write
\begin{equation*}
m_{g_0}=g_0*f_b,\quad m_g=g*f_b,
\end{equation*}
for the corresponding privatized observation densities, and denote their distribution functions by $M_{g_0}$ and $M_g$.
Our aim is to translate a discrepancy between the observable densities $m_g$ and $m_{g_0}$ into a bound on the error $W_1(g,g_0)$.

Our argument follows the frequency-splitting deconvolution strategy of \cite{rousseau2024wasserstein},
but is specialized here to the additive Laplace privatization model. 
This specialization allows us to explicitly quantify the dependence on the support radius $a$ and the noise scale $b$, allowing us to later vary privacy noise with sample size $n$. 
The proof proceeds by first smoothing the latent distribution functions with a kernel $k$ at bandwidth $h$. This introduces a bias term, but one that can be controlled by $h$. Importantly, after smoothing, we can work with the Fourier transforms of the latent distributions. 
We then separately control the low- and high-frequency parts of the Fourier transform of the difference. 
Detailed proofs of this section are provided in Appendix~\ref{app:deconv}.

We begin with some notation. 
For a probability measure $g$ with distribution function $G$, and an integrable function $f\in L^1(\mathbb R)$, define
\begin{equation*}
(g*f)(x):=\int_{\mathbb R} f(x-y)g(\mrd y),\quad
(G*f)(x):=\int_{\mathbb R} G(x-y)f(y)\mrd y.
\end{equation*}
Thus, the distribution functions $M_{g_0}$ and $M_g$ satisfy
$M_{g_0}=G_0*f_b$ and $M_g=G*f_b$.
For any $f\in L^1(\mathbb R)$, define its Fourier transform and inverse Fourier transform by
\begin{equation*}
(\mathfrak F f)(t):=\int_{\mathbb R}e^{itx}f(x)\mrd x,\quad
(\mathfrak F^{-1}\varphi)(x):=\frac{1}{2\pi}\int_{\mathbb R}e^{-itx}\varphi(t)\mrd t.
\end{equation*}
For notational convenience, we write $\FT{f}:=\mathfrak F f$.
Note that the Laplace density $f_b(x)$ has Fourier transform
$\FT{f_b}(t)=(1+b^2t^2)^{-1}$.
Fix a smoothing kernel
$k\in L^1(\mathbb R)\cap L^2(\mathbb R)$ that is even and satisfies
\begin{equation*}
\int_{\mathbb R}k(x) \mrd x=1, \quad \int_{\mathbb R}|x| |k(x)| \mrd x<\infty,
\end{equation*}
and whose Fourier transform $\FT{k}$ is compactly supported. 
For $h>0$, we define the kernel $k_h$ with bandwidth $h$ as
\begin{equation*}
k_h(x):=\frac{1}{h}k(x/h), \quad \FT{k_h}(t)=\FT{k}(ht).
\end{equation*}

In the Fourier domain, we can express $G*k_h$, the kernel-smoothed distribution function at bandwidth $h$, in terms of the privatized distribution function $M_g=G*f_b$ as
\begin{equation*}
\FT{G*k_h}(t)=\FT{M_g}(t) (1+b^2t^2)\FT{k}(ht) := \FT{M_g}(t) w_{h,b}(t) .
\end{equation*}
Let $\chi:\mathbb R\to\mathbb R$ be an even function in $C^1$ satisfying
\begin{equation*}
0\le \chi\le 1, \quad \chi(t)=1 \ \text{for } |t|\le 1, \quad \chi(t)=0 \ \text{for } |t|\ge 2.
\end{equation*}
We now split the multiplier $w_{h,b}(t)$ into low- and high-frequency parts using $\chi$. Write
\begin{equation*}
w_{1,h,b}(t):=\FT{k}(ht)\chi(bt)(1+b^2t^2), \quad w_{2,h,b}(t):=\FT{k}(ht)\left(1-\chi(bt)\right)(1+b^2t^2).
\end{equation*}
Define the terms
\begin{align*}
k_{j,h,b} &:=\mathfrak F^{-1}(w_{j,h,b}),\ j=1,2,\quad \text{and} \\
K_{2,h,b}(x) &:=\int_{-\infty}^x k_{2,h,b}(u)\mrd u.
\end{align*}
Finally, denote the smoothing bias by
$B_{g,h}:=G-G*k_h.$ \\
The following decomposition is our starting point.

\begin{restatable}[Frequency-split decomposition]{proposition}{restatefrequencysplit}\label{prop:freq-split}
For every $h>0$,
\begin{equation}\label{eq:FGKh-decomp}
G*k_h = M_g*k_{1,h,b} + m_g*K_{2,h,b}.
\end{equation}
Consequently, 
\begin{equation}\label{eq:FG-diff-decomp}
G-G_0=B_{g,h}-B_{g_0,h}+(M_g-M_{g_0})*k_{1,h,b}+(m_g-m_{g_0})*K_{2,h,b}.
\end{equation}
\end{restatable}
To turn \eqref{eq:FG-diff-decomp} into a Wasserstein bound, we use the one-dimensional identity
\begin{equation}\label{eq:w1-cdf-identity}
W_1(g,g_0)=\int_{\mathbb R}|G(x)-G_0(x)| \mathrm d x.
\end{equation}
The next lemma bounds the 1-Wasserstein distance with four terms corresponding to the bias terms $\|B_{g,h}\|_1$ and $\|B_{g_0,h}\|_1$, the low-frequency term $\|M_g-M_{g_0}\|_1 \|k_{1,h,b}\|_1$, and the high-frequency term $\|m_g-m_{g_0}\|_1 \|K_{2,h,b}\|_1$.

\begin{lemma}[Basic inversion inequality, adapted from \cite{rousseau2024wasserstein}]\label{lem:basic-inversion}
For any $g,g_0\in\mathcal P([-a, a])$ and any $h>0$, 
\begin{align*}
W_1(g,g_0) &\le \|B_{g,h}\|_1+\|B_{g_0,h}\|_1 + \|M_g-M_{g_0}\|_1 \|k_{1,h,b}\|_1 + \|m_g-m_{g_0}\|_1 \|K_{2,h,b}\|_1.
\end{align*}
\end{lemma}
\begin{proof}[Proof of Lemma~\ref{lem:basic-inversion}]
If $u,v\in L^1(\mathbb R)$, then Young's convolution inequality~\citep[Theorem 3.9.4]{bogachev2007measure} says that $u*v\in L^1(\mathbb R)$ and $\|u*v\|_1 \le \|u\|_1 \|v\|_1$.
Combining with \eqref{eq:w1-cdf-identity} yields the result.
\end{proof}
We next bound each term on the right-hand side in Lemmas~\ref{lem:bias-bound} to~\ref{lem:L1-to-W1-released}, giving an overall bound in 1-Wasserstein distance in Theorem~\ref{thm:deconv-main}. 
We start with the smoothing bias $\|B_{g,h}\|_1$.

\begin{restatable}[Smoothing bias bound]{lemma}{restatebiasbound}\label{lem:bias-bound}
For every probability measure $g$ on $\mathbb R$ and every $h>0$,
\begin{equation}\label{eq:bias-bound}
\|B_{g,h}\|_1 \le C_k h,\ \text{where}\ C_k:=\int_{\mathbb R}|x| |k(x)| \mrd x<\infty.
\end{equation}
\end{restatable}

The next two lemmas are scaled versions of the corresponding unit-scale estimates in \cite[Appendix~S1]{rousseau2024wasserstein}. To pass from the unit-scale argument to the present $\mathrm{Lap}(0,b)$ setting, we set $\eta=h/b$ and $u=bt$. This change of variables leaves the $L^1$ norm of the low-frequency kernel invariant, but contributes an additional factor of $b$ to the $L^1$ norm of the high-frequency. 
We leave the technical details in Appendix~\ref{app:deconv}.

\begin{restatable}[Low-frequency term]{lemma}{restatelowfreq}\label{lem:K1-bound}
There exists a constant $C_1(k,\chi)<\infty$, depending only on $k$ and $\chi$, such that for all sufficiently small $h/b$,
\begin{equation}\label{eq:K1-bound}
\|k_{1,h,b}\|_1\le C_1(k,\chi).
\end{equation}
\end{restatable}

\begin{restatable}[High-frequency term]{lemma}{restatehighfreq}\label{lem:F2-bound}
There exists a constant $C_2(k,\chi)<\infty$, depending only on $k$ and $\chi$, such that for all sufficiently small $h/b$, 
\begin{equation}\label{eq:F2-bound}
\|K_{2,h,b}\|_1 \le C_2(k,\chi) b^2 |\log(h/b)| h^{-1}.
\end{equation}
\end{restatable}

It remains to control the privatized distribution term $\|M_g-M_{g_0}\|_1$ by the privatized observation density error $\|m_g-m_{g_0}\|_1$.

\begin{restatable}[From density error to distribution function error]{lemma}{restatefromlone}\label{lem:L1-to-W1-released}
Let $g_1,g_2\in\mathcal P([-a,a])$. Write $m_{g_i}=g_i*f_b$ and let $M_{g_i}$ be the distribution function of $m_{g_i}$ for $i=1,2$. Set
$d:=\|m_{g_1}-m_{g_2}\|_{1}$.
Then, for every $R\ge a$,
\begin{equation}\label{eq:W1-trunc}
\|M_{g_1}-M_{g_2}\|_1 \le Rd + 2b e^{-(R-a)/b}.
\end{equation}
In particular, if $d\in(0,1)$ and
$R=a+b\log(1/d)$,
then
\begin{equation}\label{eq:L1-to-W1-released-final}
\|M_{g_1}-M_{g_2}\|_1 \le d\left(a+b\log(1/d)+2b\right)
\lesssim \left(a+b\log(1/d)\right)d.
\end{equation}
\end{restatable}
We are now ready to combine these into the main result of this section.

\begin{theorem}[Deconvolution inequality]\label{thm:deconv-main}
For any $g,g_0\in\mathcal P([-a,a])$, write $d:=\|m_g-m_{g_0}\|_1$.
If $h\in(0,b/2)$ and $d\in(0,1)$, then
\begin{align}
W_1(g,g_0)&\lesssim h + \|M_g-M_{g_0}\|_1 + b^2 |\log(h/b)| h^{-1} d \notag\\
&\lesssim h + \left(a+b\log(1/d)\right)d + b^2 |\log(h/b)| h^{-1}d.
\label{eq:deconv-main-2}
\end{align}
\end{theorem}

\begin{proof}[Proof of Theorem~\ref{thm:deconv-main}]
Combine Lemma~\ref{lem:basic-inversion} with Lemmas~\ref{lem:bias-bound}, \ref{lem:K1-bound}, and \ref{lem:F2-bound} to obtain
\begin{equation*}
W_1(g,g_0) \lesssim h + \|M_g-M_{g_0}\|_1 + b^2 |\log(h/b)| h^{-1} \|m_g-m_{g_0}\|_1.
\end{equation*}
This proves the first inequality. The second inequality follows by Lemma~\ref{lem:L1-to-W1-released}.
\end{proof}

Theorem~\ref{thm:deconv-main} completes the task of bounding the Wasserstein distance for a fixed $h$.
The next result chooses the optimal bandwidth by balancing the bias and the high-frequency terms.

\begin{corollary}[Optimized bandwidth choice]\label{cor:deconv-optimized}
Under the assumptions of Theorem~\ref{thm:deconv-main}, let $d:=\|m_g-m_{g_0}\|_1$.
There exists a constant $d_0\in(0,1)$ such that, whenever $d\in(0,d_0]$, choosing $h \asymp b\sqrt{d\log(1/d)}$ ensures that $h<b/2$, and
\begin{equation}\label{eq:deconv-optimized}
W_1(g,g_0) \lesssim ad + b\sqrt{d\log(1/d)}.
\end{equation}
\end{corollary}

\begin{proof}[Proof of Corollary~\ref{cor:deconv-optimized}]
Let $h \asymp b\sqrt{d\log(1/d)}$. Since $d\log(1/d) \rightarrow 0$ as $d \rightarrow 0$, there exists $d_0 >0$ such that $h<b/2$ whenever $d \le d_0$.
Moreover, $h/b \asymp \sqrt{d\log(1/d)}$, so that
$|\log(h/b)|\asymp \log(1/d)$
for sufficiently small $d$. Therefore
\begin{equation*}
b^2 |\log(h/b)| h^{-1} d \asymp b \sqrt{d \log(1 / d)}.
\end{equation*}
Moreover, note that $bd\log(1/d)\le b\sqrt{d\log(1/d)}$ holds for all sufficiently small $d$, so the term $\left(a+b\log(1/d)\right)d$ is bounded by $ad + b\sqrt{d\log(1/d)}$.
Combining these bounds to the terms in~\eqref{eq:deconv-main-2} proves \eqref{eq:deconv-optimized}.
\end{proof}

\section{Rate of convergence of the NPMLE}\label{sec:rate}

In this section, we seek to establish rates of convergence of the NPMLE $\widehat g_n$ defined in~\eqref{eq:npmle} to the true distribution $g_0$ in the 1-Wasserstein sense.
Section~\ref{sec:deconv} established that for any competing mixing distribution $g\in\mathcal P([-a,a])$, the latent error $W_1(g,g_0)$ can be controlled in terms of the privatized observation density error $\|m_g-m_{g_0}\|_1$.
We therefore need {to control the $L^1$ error of the privatized observation density $m_{\widehat g_n}$.} We first derive a Hellinger rate for $m_{\widehat g_n}$ relative to $m_{g_0}$ following the convex class maximum likelihood argument of \cite{van1996rates}; see also Remark~3 of \cite{scricciolo2018bayes}.
The Hellinger rate is then converted into an $L^1$ rate through the standard inequality recalled below.
We then plug this rate into Corollary~\ref{cor:deconv-optimized} to obtain a Wasserstein rate for the latent mixing distribution $\widehat g_n$.

Recall the squared Hellinger distance:
\begin{equation*}
H^2(p,q):=\frac12\int_{\mathbb R}\left(\sqrt p-\sqrt q\right)^2 \mrd x.
\end{equation*}
We will use the standard inequality $\|p-q\|_1\le 2\sqrt{2}H(p,q)$, so a Hellinger rate for the privatized observation density also gives an $L^1$ rate.
We then have the following result:
\begin{restatable}[Hellinger rate for the privatized observation density]{theorem}{restatehellingerrate}\label{thm:direct-rate}
Let $\widehat g_n$ be the NPMLE over $\mathcal P([-a,a])$, and suppose that the true mixing distribution $g_0$ belongs to $\mathcal P([-a,a])$.
Then the Hellinger distance between $m_{\widehat g_n}$ and $m_{g_0}$ satisfies
\begin{equation}\label{eq:direct-rate}
H(m_{\widehat g_n},m_{g_0}) = \mathcal O_p \left((1+a/b)^{1/4}e^{a/(2b)}n^{-3/8}\right).
\end{equation}
\end{restatable}
The exponential factor comes from the envelope bound used in the convex-class likelihood argument.
This bound is not expected to be tight in the small noise regime when $b$ goes to 0.
Sharper control may be possible under additional structure on $g_0$, such as smoothness of the latent density; see, for example, \cite{fan1991optimal,rousseau2024wasserstein}.
Here, we keep the analysis uniform over all bounded mixing distributions and focus on how the privacy noise scale enters the asymptotic recovery bound.
A detailed proof is provided in Appendix~\ref{app:rate}. 

We now return from the privatized density to the mixing distribution.
Combining Theorem~\ref{thm:direct-rate} with the deconvolution inequality from Section~\ref{sec:deconv} gives the corresponding $W_1$ rate.

\begin{corollary}[Rate for the latent mixing distribution]\label{cor:w1-rate}
Under the assumptions of Theorem~\ref{thm:direct-rate}, we have the $L^1$ rate for the privatized density
\begin{equation*}
\|m_{\widehat g_n}-m_{g_0}\|_1 = \mathcal O_p\left((1+a/b)^{1/4}e^{a/(2b)}n^{-3/8}\right),
\end{equation*}
and the 1-Wasserstein rate for the mixing distribution
\begin{equation}\label{eq:w1-rate}
W_1(\widehat g_n,g_0) = \mathcal O_p \left( a r_n + b\sqrt{r_n\log(1/r_n)} \right),
\quad r_n:=(1+a/b)^{1/4}e^{a/(2b)}n^{-3/8}.
\end{equation}
Equivalently,
\begin{equation*}
W_1(\widehat g_n,g_0) = \mathcal O_p\left( a(1+a/b)^{1/4}e^{a/(2b)}n^{-3/8} + b(1+a/b)^{1/8}e^{a/(4b)}n^{-3/16}\sqrt{\log n} \right).
\end{equation*}
For fixed $a$ and $b$, the second term gives the leading order
\begin{equation*}
W_1(\widehat g_n,g_0)=\mathcal O_p\left(n^{-3/16}\sqrt{\log n}\right).
\end{equation*}
\end{corollary}

\begin{proof}[Proof of Corollary~\ref{cor:w1-rate}]
The bound on $\|m_{\widehat g_n}-m_{g_0}\|_1$ follows from \eqref{eq:direct-rate} and the inequality
\begin{equation*}
\|m_{\widehat g_n}-m_{g_0}\|_1 \le 2\sqrt{2} H(m_{\widehat g_n} , m_{g_0}).
\end{equation*}
Because $r_n\to0$, We apply Corollary~\ref{cor:deconv-optimized} with $g=\widehat g_n$ and
$d=\|m_{\widehat g_n}-m_{g_0}\|_1=\mathcal O_p(r_n)$.
This yields
\begin{equation*}
W_1(\widehat g_n,g_0) = \mathcal O_p \left( a r_n + b\sqrt{r_n\log(1/r_n)} \right),
\end{equation*}
which proves \eqref{eq:w1-rate}. Since
$\log(1/r_n)=\frac{3}{8}\log n + \mathcal{O}(1)$,
the expanded form follows.
\end{proof}

To make the privacy dependence explicit, we now allow the Laplace scale $b$ to vary with the sample size $n$.
The next corollary gives a sufficient condition under which the NPMLE remains consistent even when the privacy noise increases with $n$.

\begin{restatable}[Growth condition on the privacy scale]{corollary}{restategrowthcondition}\label{cor:b-growth}
Assume $a>0$ is fixed, and let $b=b_n$ be a nondecreasing sequence.
Then for the NPMLE $\widehat g_n$, we have
\begin{equation}\label{eq:w1-rate-bn}
W_1(\widehat g_n,g_0) = \mathcal O_p \left( a n^{-3/8} + b_n n^{-3/16}\sqrt{\log n} \right).
\end{equation}
In particular, if the growth rate of $b_n$ satisfies
\begin{equation}\label{eq:b-sufficient}
b_n=o\left(n^{3/16}(\log n)^{-1/2}\right),
\end{equation}
then
$W_1(\widehat g_n,g_0)\to 0\ \mathrm{ in\ probability}$.

Equivalently, if we write the privacy loss budget as $\epsilon_n={2a}/{b_n}$, then \eqref{eq:b-sufficient} means that $\epsilon_n$ cannot decay as fast as the boundary rate $n^{-3/16}(\log n)^{1/2}$.
\end{restatable}

The next theorem provides a converse in the large noise regime. Once the Laplace scale $b_n$ grows at least on the order of $\sqrt n$, no estimator based on the privatized sample can be uniformly consistent over the class $\mathcal P([-a,a])$.

\begin{restatable}[Impossibility under $\sqrt n$-scale Laplace noise]{theorem}{restateimpossibility}\label{thm:impossibility-sqrtn}
Fix $a>0$. For each $n\ge1$, let $b_n>0$, and suppose that we observe
\begin{equation*}
X_{1,n},\dots,X_{n,n}\overset{iid}{\sim} p_{g,b_n}, \quad p_{g,b_n}=f_{b_n}*g, \quad f_{b_n}(x):=\frac1{2b_n}e^{-|x|/b_n}, \quad g\in\mathcal P([-a,a]).
\end{equation*}
Assume that there exists a constant $c>0$ such that
$b_n\ge c\sqrt n$
for all sufficiently large $n$. Then for every estimator $g_n^\star=g_n^\star(X_{1,n},\dots,X_{n,n})$,
\begin{equation}\label{eq:minimax-impossibility}
\liminf_{n\to\infty}\sup_{g\in\mathcal P([-a,a])}
\mathbb P_{g}^{n} \left( W_1(g_n^\star,g)\ge a \right) \ge \frac14\exp \left(-\frac{2a^2}{c^2}\right) >0,
\end{equation}
where $\mathbb P_{g}^{n}$ denotes the joint distribution of $(X_{1,n},\dots,X_{n,n})$ under $g$.
Consequently, there does not exist any estimator sequence $g_n^\star$ such that
\begin{equation*}
W_1(g_n^\star,g)\to 0
\quad\text{in }\mathbb P_{g}^{n}\text{-probability for every }g\in\mathcal P([-a,a]).
\end{equation*}
Equivalently, writing the privacy loss budget as $\epsilon_n={2a}/{b_n}$, if $\epsilon_n\lesssim n^{-1/2}$, then no estimator sequence can consistently recover $g_0$ in $W_1$ over the class $\mathcal P([-a,a])$.
\end{restatable}

Taken together, Corollary~\ref{cor:b-growth} and Theorem~\ref{thm:impossibility-sqrtn} identify two explicit noise regimes.
If $b_n$ grows slower than $n^{3/16}(\log n)^{-1/2}$, then the NPMLE remains consistent in $W_1$.
If $b_n$ grows on the order of $\sqrt n$ or larger, then uniform consistent recovery over $\mathcal P([-a,a])$ is impossible for any estimator.
The numerical experiments in the next section illustrate the corresponding finite-sample behavior.

\section{Numerical experiments}\label{sec:exp}

We study the empirical behavior of the NPMLE under the Laplace convolution model.
The experiments in this section have three purposes: to illustrate how the latent recovery error depends on the sample size $n$ and the privacy loss budget $\epsilon$, to examine the empirical support size of the fitted NPMLE through its active support size, and to compare the simulations with the theoretical scaling behaviors from Section~\ref{sec:rate}.

\paragraph{Experimental setup.}We consider three distributions $g_0$ on $[-a,a]$. 
The first is a sparse discrete distribution,
\begin{equation*}
g_0^{\mathrm{disc}}=0.25\delta_{-0.80a}+0.45\delta_{-0.05a}+0.30\delta_{0.65a}.
\end{equation*}
The second is a continuous distribution $g_0^{\beta}$ obtained by the mapping 
\begin{equation*}
\theta=-a+2aU, \quad U\sim\Beta(2.5,5).
\end{equation*}
The third is a mixed beta-atom distribution,
\begin{equation*}
g_0^{\mathrm{mix}}=0.2\delta_{-a}+0.3H_L+0.3 H_R+0.2 \delta_a,
\end{equation*}
where $H_L$ is the distribution of $-0.7a + 0.5a V_L$ with $V_L\sim\Beta(2,6)$, so that $H_L$ is supported on $[-0.7a,-0.2a]$, and $H_R$ is the distribution of $0.2a + 0.8a V_R$ with $V_R\sim\Beta(6,2)$, so that $H_R$ is supported on $[0.2a,a]$.
Throughout the experiments, we set $a=3$ and use the relation $b=2a/\epsilon$ when the privacy budget $\epsilon$ is varied.

For each combination of $n\in\{50,100,200,500,1000,2000\}$ and $\epsilon\in\{0.1,0.2,0.5,1, 2,\allowbreak 5, 10, 20\}$, we compute the NPMLE $\widehat g_n$ by Algorithm~\ref{alg:em} and evaluate the 1-Wasserstein error $W_1(\widehat g_n,g_0)$. 
Each configuration is repeated 50 times, and we report the empirical mean together with one empirical standard deviation.

\begin{figure}[htb]
    \centering
    \includegraphics[width=\textwidth]{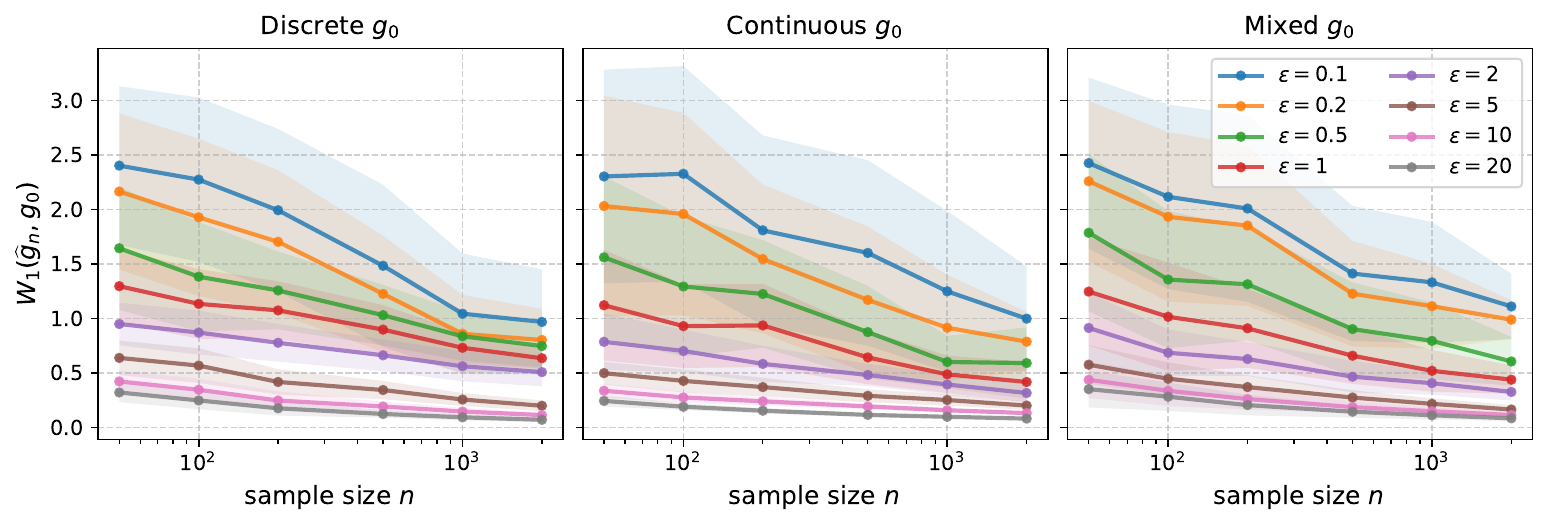}
    \vspace{-10pt}
    \caption{Wasserstein error $W_1(\widehat g_n,g_0)$ versus sample size $n$, for several privacy loss budgets $\epsilon$. 
    The shaded bands indicate one empirical standard deviation over 50 repetitions.}
    \label{fig:w1-n}
\end{figure}

\begin{figure}[htb]
    \centering
    \includegraphics[width=\textwidth]{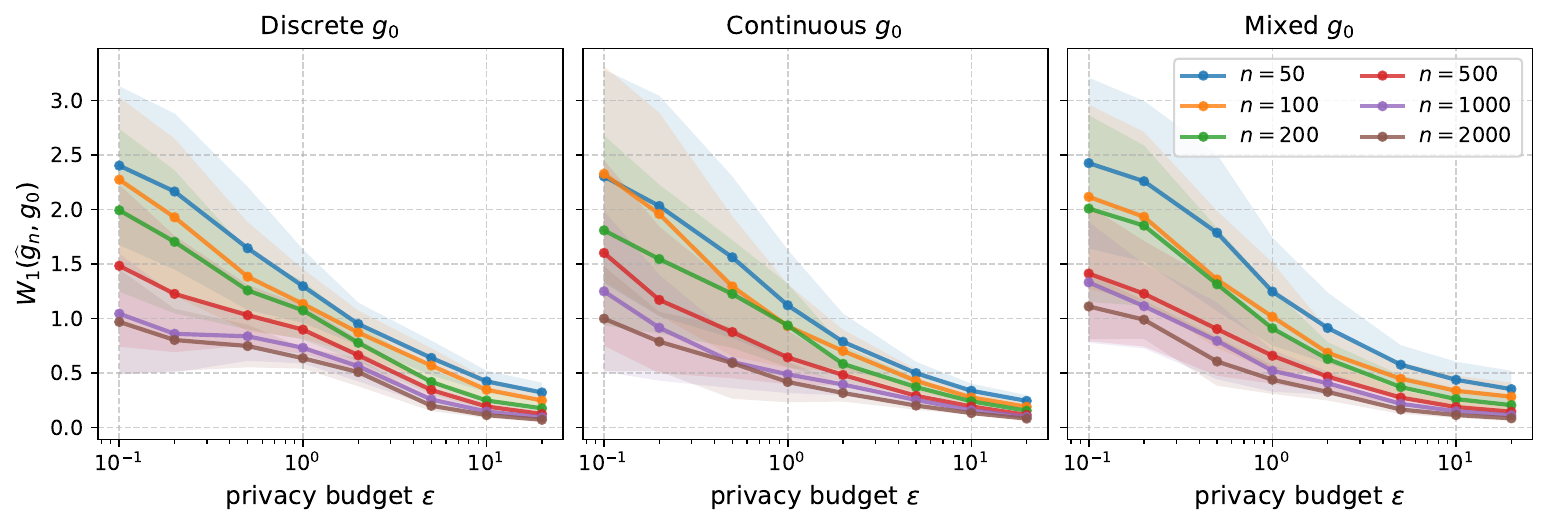}
    \vspace{-10pt}
    \caption{Wasserstein error $W_1(\widehat g_n,g_0)$ versus privacy loss budget $\epsilon$, for several sample sizes $n$. 
    The shaded bands indicate one empirical standard deviation over 50 repetitions.}
    \label{fig:w1-eps}
\end{figure}

\paragraph{Visualization of latent recovery error.}Figures~\ref{fig:w1-n} and~\ref{fig:w1-eps} summarize the recovery behavior. 
Figure~\ref{fig:w1-n} plots $W_1(\widehat g_n,g_0)$ against $n$ for several privacy loss budgets $\epsilon$, while Figure~\ref{fig:w1-eps} plots the same quantity against $\epsilon$ for several sample sizes $n$. 
Across all three choices of $g_0$, the error decreases as $n$ increases and also decreases as $\epsilon$ increases, in agreement with the qualitative prediction of Corollary~\ref{cor:w1-rate}. 
{The decrease is especially visible when moving from very small privacy budgets to moderate values of $\epsilon$, while the curves tend to flatten for larger $\epsilon$.
The three latent distributions show broadly similar patterns. We also find that the discrete case is not uniformly easier than the continuous or mixed cases.}

\paragraph{Active support size of the fitted NPMLE.}In addition to recovery error, we also examine the active support size of the fitted NPMLE.
In the Gaussian-noise setting, \cite{polyanskiy2020self} showed that the NPMLE has support size of order $\mathcal{O}(\log n)$  with high probability.
We are not aware of an analogous result for Laplace mixtures, and we examine this behavior empirically.
In our results, we report a thresholded active support size,
\begin{equation*}
N_{\mathrm{act},\tau}(\widehat g_n):=\#\left\{j:\widehat w_j>\tau\right\},
\end{equation*}
where $\tau=10^{-10}$ is a fixed threshold used uniformly across all experiments.
The thresholding is needed because the EM updates cannot drive positively initialized weights exactly to zero after a finite number of steps.

\begin{figure}[htbp]
    \centering
    \includegraphics[width=\textwidth]{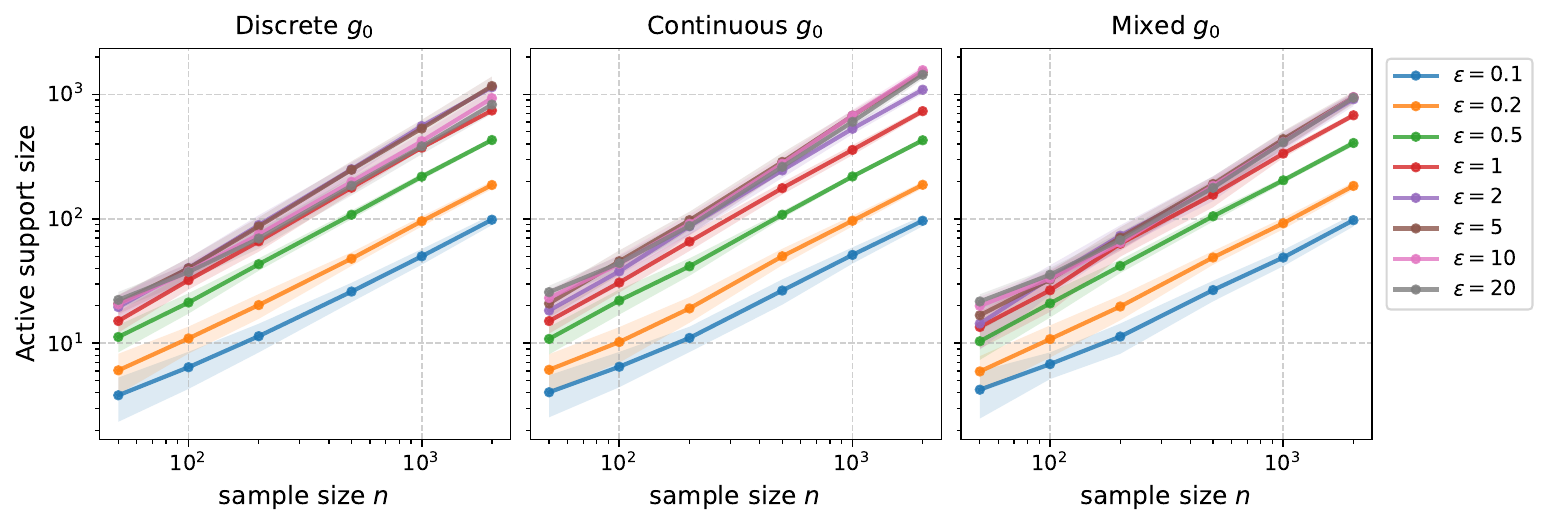}
    \vspace{-10pt}
    \caption{Log-log plot for active support size versus sample size $n$, for several privacy loss budgets $\epsilon$. The vertical axis is shown on a logarithmic scale, and the shaded bands indicate one empirical standard deviation over 50 repetitions.}
    \label{fig:act-comp}
\end{figure}

Figure~\ref{fig:act-comp} shows that the active support size increases with both $n$ and $\epsilon$ across all three choices of $g_0$.
Even when the true mixing distribution is discrete with only three atoms, the fitted NPMLE under Laplace noise typically uses many active support points, and this number grows steadily with the sample size.
Within the sample size range considered in this paper, the growth appears substantially faster than logarithmic, with the log-log plots suggesting a polynomial order increase in $n$.
We do not interpret these experiments as identifying a precise asymptotic law, but they do suggest that the fitted support size in the Laplace case behaves quite differently from the logarithmic regime established in the Gaussian case.

\paragraph{Scaling of privacy noise.}We next consider a regime in which the privacy noise increases with the sample size according to
\begin{equation*}
b_n=\left(\frac{n}{n_0}\right)^\alpha.
\end{equation*}
Corollary~\ref{cor:b-growth} shows that the NPMLE remains consistent as long as $b_n$ grows no faster than $n^{3/16}(\log n)^{-1/2}$. 
Motivated by this result, we examine several polynomial growth exponents $\alpha$ around the theoretical scale $3/16$.

Figure~\ref{fig:bn-growth} displays the mean Wasserstein error $W_1(\widehat g_n,g_0)$ for the three choices of $g_0$ over 50 replications, where we set $n_0 = 50$.
When $\alpha\le3/16$, the error decreases steadily with $n$ in all three cases, which is consistent with the convergence guarantee. 
{For growth rates moderately above this boundary, the decrease becomes slower and less uniform across the three latent distributions.
At $\alpha=1/2$, corresponding to the $\sqrt n$ noise scale in Theorem~\ref{thm:impossibility-sqrtn}, the curves no longer show a clear decay pattern, while for still larger values such as $\alpha=0.75$ and $\alpha=1$, the error often increases with $n$.
These finite-sample results do not identify a sharp transition, but they show a clear deterioration as the noise growth approaches the $\sqrt n$ scale.}

\begin{figure}[htbp]
    \centering
    \includegraphics[width=\textwidth]{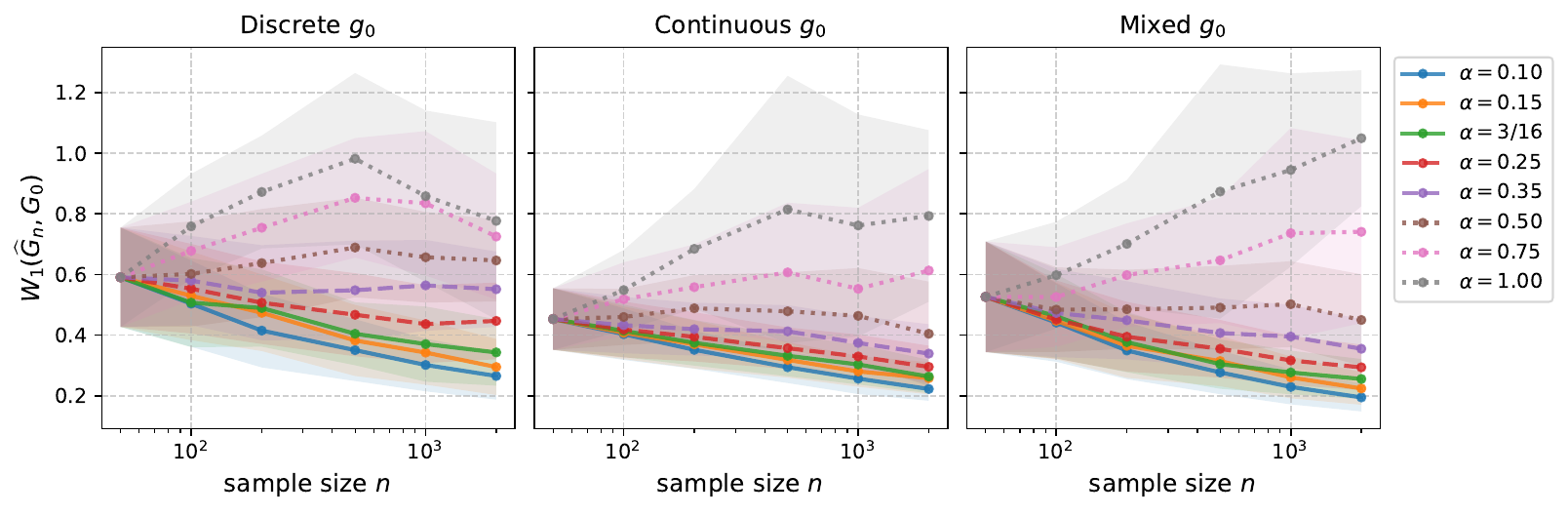}
    \vspace{-10pt}
    \caption{Wasserstein error $W_1(\widehat g_n,g_0)$ under the growing-noise regime for different growth rates $\alpha$. Solid lines indicate the sufficient consistency regime $\alpha\le 3/16$, dashed lines indicate intermediate growth rates, and dotted lines indicate $\alpha\ge 1/2$.}
    \label{fig:bn-growth}
\end{figure}

\section{Conclusion}\label{sec:conclusion}

We studied nonparametric maximum likelihood estimation for recovering an unknown latent distribution from locally differentially private observations generated by the Laplace mechanism.
We showed that the NPMLE admits a finite-dimensional formulation where the support is restricted to a projected observation set, leading to a practical EM algorithm. 
Building on the Wasserstein inversion framework of~\citep{rousseau2024wasserstein}, we derived a Laplace-specific inversion inequality that bounds the Wasserstein error between latent distributions in terms of the error in the privatized observation density, while explicitly tracking its dependence on the Laplace noise scale.
Together, these results provide a likelihood-based framework for latent distribution recovery under Laplace local differential privacy and make the associated privacy-utility tradeoff explicit.

Our theoretical analysis shows that the NPMLE remains consistent whenever the privacy loss budget $\epsilon_n$ decays slower than $n^{-3/16}(\log n)^{1/2}$.
In contrast, if $\epsilon_n$ decays faster than $n^{-1/2}$, uniformly consistent recovery over $\mathcal P([-a,a])$ is impossible for any estimator.
We empirically study finite-sample behavior across different sample sizes and privacy levels.

Several questions remain open.
The most immediate is to close the gap between the $n^{-3/16}$ rate sufficient for consistency and the $n^{-1/2}$ rate necessary for consistency.
Another important direction is to sharpen the support size of the NPMLE under the Laplace convolution model.
For Gaussian mixtures, logarithmic support-size bounds are known~\citep{polyanskiy2020self}, whereas no comparable result is currently available for the Laplace case. 
Our empirical results suggest that the support size grows faster than logarithmically, indicating a difference between the Gaussian and Laplace settings.

Finally, it is of interest to extend the analysis beyond bounded latent spaces and beyond one dimension. 
The multidimensional setting presents significant new challenges. Although the support of the latent distribution remains restricted to a finite candidate set determined by the observations, this candidate set is the Cartesian product of the coordinate-wise projected observations and therefore grows exponentially with the dimension. %
Moreover, unlike in one-dimension, the Wasserstein distance is no longer just the $L^1$ distance between distribution functions, so that asymptotic analysis in this setting will require new techniques.

\section*{Acknowledgments}
VR acknowledges support from the National Science Foundation through grant DMS-2503118.
NJ acknowledges support from the National Science Foundation through grant DMS-2610236.

\bibliographystyle{plainnat}
\bibliography{refs}

\begin{appendix}

\section{Proofs for Section~\ref{sec:npmle}}\label{app:npmle}

\restatesuppportpw*

\begin{proof}[Proof of Lemma~\ref{lem:support-characterization}]
Fix $\mu\in[-a,a]$ and consider $g_t=(1-t)\widehat g_n+t\delta_\mu$ for $t\in[0,1]$, where $\delta_{\mu}$ denotes the Dirac probability measure concentrated at $\mu$. By optimality of $\widehat g_n$,
\begin{equation*}
\ell_n(\widehat g_n) \ge \ell_n(g_t).
\end{equation*}
Hence, the right derivative at $t=0$ satisfies
\begin{align*}
0 &\ge \left.\frac{\mrd}{\mrd t}\ell_n(g_t)\right|_{t=0}\\
&=\left.\frac{\mrd}{\mrd t}\frac{1}{n}\sum_{i=1}^n \log\left((1-t)m_{\widehat g_n}(X_i)+t m_{\delta_\mu}(X_i)\right)\right|_{t=0}\\
&=\frac{1}{n}\sum_{i=1}^n \frac{m_{\delta_\mu}(X_i)-m_{\widehat g_n}(X_i)}{m_{\widehat g_n}(X_i)} \\
&=\left(\frac{1}{n}\sum_{i=1}^n \frac{m_{\delta_\mu}(X_i)}{m_{\widehat g_n}(X_i)} \right) - 1,
\end{align*}
and note that
\begin{equation*}
m_{\delta_\mu}(X_i) = \int f_b(X_i-\theta)\delta_\mu(\mathrm{d}\theta)=f_b(X_i-\mu),
\end{equation*}
that is,
\begin{equation}\label{eq:Deqle1}
D_{\widehat g_n}(\mu) =\frac{1}{n}\sum_{i=1}^n \frac{f_b(X_i-\mu)}{m_{\widehat g_n}(X_i)} \le 1 \quad \text{for all }\mu\in[-a,a].
\end{equation}
On the other hand, averaging $D_{\widehat g_n}(\cdot)$ with respect to $\widehat g_n(\cdot)$ gives
\begin{align}\label{eq:avg_d_on_optim_measure}
\int D_{\widehat g_n}(\mu)\widehat g_n(\mathrm{d}\mu)
&= \int \left(\frac{1}{n}\sum_{i=1}^n \frac{f_b(X_i - \mu)}{m_{\widehat g_n}(X_i)}\right)\widehat g_n(\mathrm{d}\mu) \notag\\
&= \frac{1}{n}\sum_{i=1}^n \frac{\int f_b(X_i - \mu) \widehat g_n(\mathrm{d}\mu)}{m_{\widehat g_n}(X_i)} \\
&= \frac{1}{n}\sum_{i=1}^n 1 = 1.
\end{align}
From Eq.~\eqref{eq:Deqle1}, the global maximum of $D_{\widehat g_n}(\cdot)$ is at most 1. From Eq.~\eqref{eq:avg_d_on_optim_measure}, we know its average value under $\widehat g_n$ is exactly 1. This is only possible if the measure $\widehat g_n$ places all its mass on the set where $D_{\widehat g_n}(\mu)$ achieves its maximum value of 1. Therefore,
\begin{equation*}
\supp(\widehat g_n) \subseteq \{\mu \in [-a,a] : D_{\widehat g_n}(\mu) = 1\} = \arg\max_{\mu\in[-a,a]} D_{\widehat g_n}(\mu).
\end{equation*}
\end{proof}

\restategeneralkernel*

\begin{proof}[Proof of Corollary~\ref{cor:general-kernel-support}]
By Lemma~\ref{lem:support-characterization}, it suffices to show that every maximizer of
\begin{equation*}
D_{\widehat g_n}(\mu):=\frac{1}{n}\sum_{i=1}^n \frac{f(X_i-\mu)}{m_{\widehat g_n}(X_i)},\quad \mu\in[a_1,a_2],
\end{equation*}
belongs to
\begin{equation*}
\left\{\Pi_{[a_1,a_2]}(X_1),\dots,\Pi_{[a_1,a_2]}(X_n)\right\}.
\end{equation*}
Write
\begin{equation*}
c_i:=\frac{1}{n m_{\widehat g_n}(X_i)}>0,\quad i=1,\dots,n.
\end{equation*}
Since $f(x-\mu)=\varphi(|x-\mu|)$,
\begin{equation*}
D_{\widehat g_n}(\mu)=\sum_{i=1}^n c_i \varphi(|X_i-\mu|).
\end{equation*}

Let $u_1<\cdots<u_{\nu}$ be the distinct values among
$\Pi_{[a_1,a_2]}(X_1),\dots,\Pi_{[a_1,a_2]}(X_n)$.
Consider any open interval
\begin{equation*}
I\subset [a_1,a_2]\setminus\{u_1,\dots,u_{\nu}\}.
\end{equation*}
Then for each $i$, the sign of $X_i-\mu$ is constant on $I$, so $\mu\mapsto |X_i-\mu|$ is affine on $I$. Hence, for every $\mu\in I$,
\begin{equation*}
\frac{\mathrm d^2}{\mathrm d\mu^2}\varphi(|X_i-\mu|)=\varphi''(|X_i-\mu|)>0,
\end{equation*}
and therefore
\begin{equation*}
\frac{\mathrm d^2}{\mathrm d\mu^2}D_{\widehat g_n}(\mu)=\sum_{i=1}^n c_i \varphi''(|X_i-\mu|)>0.
\end{equation*}
Thus $D_{\widehat g_n}$ is strictly convex on every connected component of
$[a_1,a_2]\setminus\{u_1,\dots,u_{\nu}\}$ contained in $(a_1,a_2)$, so it cannot attain its maximum at an interior point of such a component.

It remains to treat the boundary points. If $a_1\notin\{u_1,\dots,u_{\nu}\}$, then $X_i>a_1$ for all $i$, so for $\mu\in [a_1,u_1)$,
\begin{equation*}
|X_i-\mu|=X_i-\mu\quad\text{and}\quad\frac{\mathrm d}{\mathrm d\mu}D_{\widehat g_n}(\mu)=-\sum_{i=1}^n c_i \varphi'(X_i-\mu)>0,
\end{equation*}
since $\varphi'(t)<0$ for $t>0$. Hence $D_{\widehat g_n}$ is strictly increasing on $[a_1,u_1)$, and $a_1$ cannot be a maximizer. Similarly, if $a_2\notin\{u_1,\dots,u_{\nu}\}$, then $X_i<a_2$ for all $i$, so for $\mu\in (u_{\nu},a_2]$,
\begin{equation*}
|X_i-\mu|=\mu-X_i\quad\text{and}\quad\frac{\mathrm d}{\mathrm d\mu}D_{\widehat g_n}(\mu)=\sum_{i=1}^n c_i \varphi'(\mu-X_i)<0.
\end{equation*}
Thus $D_{\widehat g_n}$ is strictly decreasing on $(u_{\nu},a_2]$, and $a_2$ cannot be a maximizer.

Therefore every maximizer of $D_{\widehat g_n}$ over $[a_1,a_2]$ must belong to
\begin{equation*}
\{u_1,\dots,u_{\nu}\}=\left\{\Pi_{[a_1,a_2]}(X_1),\dots,\Pi_{[a_1,a_2]}(X_n)\right\}.
\end{equation*}
Applying Lemma~\ref{lem:support-characterization} gives
\begin{equation*}
\supp(\widehat g_n)\subseteq \left\{\Pi_{[a_1,a_2]}(X_1),\dots,\Pi_{[a_1,a_2]}(X_n)\right\}.
\end{equation*}
\end{proof}

\begin{lemma}[Full column rank of the Laplace likelihood matrix]\label{lem:laplace-full-rank}
Let $u_1<\cdots<u_\nu$ be the distinct values among
\begin{equation*}
\Pi_{[-a,a]}(X_1),\dots,\Pi_{[-a,a]}(X_n),
\end{equation*}
and define the $n\times \nu$ matrix
\begin{equation*}
\mathbf S=(s_{ij}),\quad s_{ij}:=f_b(X_i-u_j)=\frac{1}{2b}\exp\left(-\frac{|X_i-u_j|}{b}\right).
\end{equation*}
Then the columns of $\, \mathbf S$ are linearly independent.
\end{lemma}

\begin{proof}[Proof of Lemma~\ref{lem:laplace-full-rank}]
Suppose that
\begin{equation*}
\sum_{j=1}^\nu c_j s_{ij}=0,\quad i=1,\dots,n,
\end{equation*}
for some coefficients $c_1,\dots,c_\nu$. Define
\begin{equation*}
f(x):=\sum_{j=1}^\nu \frac{c_j}{2b}\exp\left(-\frac{|x-u_j|}{b}\right).
\end{equation*}
Then $f(X_i)=0$ for all $i=1,\dots,n$.
Let $X_{(1)}\le \cdots \le X_{(n)}$ denote the ordered observations. Since $X_{(1)}\le u_1$, evaluating at $X_{(1)}$ gives
\begin{equation}\label{eq:left-tail-identity}
\sum_{j=1}^\nu c_j e^{-u_j/b}=0.
\end{equation}

We now eliminate the coefficients successively. For $k=1,\dots,\nu-2$, the point $u_{k+1}$ lies in $(-a,a)$, so it must equal one of the observations. Hence there exists an index $t_k$ such that
\begin{equation*}
X_{t_k}=u_{k+1}\in (u_k,u_{k+1}].
\end{equation*}
Assume inductively that $c_1=\cdots=c_{k-1}=0$. Then \eqref{eq:left-tail-identity} reduces to
\begin{equation*}
\sum_{j=k}^\nu c_j e^{-u_j/b}=0.
\end{equation*}
Evaluating $f$ at $X_{t_k}=u_{k+1}$ yields
\begin{align*}
0&=\sum_{j=k}^\nu \frac{c_j}{2b}\exp\left(-\frac{|u_{k+1}-u_j|}{b}\right) \\
&=\frac{c_k}{2b}\exp\left(\frac{u_k-u_{k+1}}{b}\right)+\frac{1}{2b}\exp\left(\frac{u_{k+1}}{b}\right)\sum_{j=k+1}^\nu c_j e^{-u_j/b} \\
&=\frac{c_k}{2b}\exp\left(\frac{u_k-u_{k+1}}{b}\right)-\frac{c_k}{2b}\exp\left(\frac{u_{k+1}-u_k}{b}\right) \\
&=-\frac{c_k}{b}\sinh\left(\frac{u_{k+1}-u_k}{b}\right).
\end{align*}
Since $u_{k+1}>u_k$, this implies $c_k=0$. Thus
\begin{equation*}
c_1=\cdots=c_{\nu-2}=0.
\end{equation*}

It remains to treat the last two coefficients $c_{\nu-1}$ and $c_\nu$.

\textbf{Case (i).} If $u_\nu<a$, or if $u_\nu=a$ and some observation equals $a$, then there exists an observation in $(u_{\nu-1},u_\nu]$. Repeating the same argument once more yields $c_{\nu-1}=0$, and then~\eqref{eq:left-tail-identity} gives $c_\nu=0$.

\textbf{Case (ii).} If $u_\nu=a$ and no observation lies in $(u_{\nu-1},a]$. Then there must exist an observation $X_r>a$, since otherwise the projected value $a$ would not occur among
\begin{equation*}
\Pi_{[-a,a]}(X_1),\dots,\Pi_{[-a,a]}(X_n).
\end{equation*}
Using $c_1=\cdots=c_{\nu-2}=0$, equation \eqref{eq:left-tail-identity} becomes
\begin{equation}\label{eq:last-two-left}
c_{\nu-1}e^{-u_{\nu-1}/b}+c_\nu e^{-a/b}=0.
\end{equation}
On the other hand, evaluating $f(X_r)=0$ and using $X_r>a=u_\nu>u_{\nu-1}$ gives
\begin{equation}\label{eq:last-two-right}
c_{\nu-1}e^{u_{\nu-1}/b}+c_\nu e^{a/b}=0.
\end{equation}
From \eqref{eq:last-two-left}, we have
\begin{equation*}
c_\nu=-c_{\nu-1}\exp\left(\frac{a-u_{\nu-1}}{b}\right).
\end{equation*}
Substituting this into \eqref{eq:last-two-right} gives
\begin{align*}
0&=c_{\nu-1}e^{u_{\nu-1}/b}+c_\nu e^{a/b} \\
&=c_{\nu-1}e^{u_{\nu-1}/b}-c_{\nu-1}\exp\left(\frac{a-u_{\nu-1}}{b}\right)e^{a/b} \\
&=c_{\nu-1}\left(e^{u_{\nu-1}/b}-e^{(2a-u_{\nu-1})/b}\right) \\
&=-2c_{\nu-1}e^{a/b}\sinh\left(\frac{a-u_{\nu-1}}{b}\right).
\end{align*}
Since $a>u_{\nu-1}$, we have
\begin{equation*}
\sinh\left(\frac{a-u_{\nu-1}}{b}\right)>0,
\end{equation*}
and therefore $c_{\nu-1}=0$. Then \eqref{eq:last-two-left} yields $c_\nu=0$.

Therefore, $c_1=\cdots=c_\nu=0$, and the columns of $\mathbf S$ are linearly independent.
\end{proof}

\restatestrictconcavity*
\begin{proof}[Proof of Corollary~\ref{cor:strict-concavity}]
Write
\begin{equation*}
s_{ij}:=f_b(X_i-u_j), \quad i=1,\dots,n,\ \ j=1,\dots,\nu.
\end{equation*}
Then
\begin{equation*}
\ell_n(w)=\sum_{i=1}^n \log \left(\sum_{j=1}^{\nu} w_j s_{ij}\right).
\end{equation*}
Its Hessian is given by
\begin{equation*}
\frac{\partial^2 \ell_n(w)}{\partial w_j\partial w_k}=-\sum_{i=1}^n\frac{s_{ij}s_{ik}}{\left(\sum_{\ell=1}^\nu w_\ell s_{i\ell}\right)^2}.
\end{equation*}
Hence, for any $v=(v_1,\dots,v_\nu)^\top\in\mathbb R^\nu$,
\begin{align*}
v^\top \nabla^2 \ell_n(w)\, v
&=-\sum_{i=1}^n\frac{\left(\sum_{j=1}^\nu v_j s_{ij}\right)^2}{\left(\sum_{\ell=1}^\nu w_\ell s_{i\ell}\right)^2}\le 0.
\end{align*}
Therefore $\ell_n$ is concave on $\Delta_\nu$.

If equality holds, then
\begin{equation*}
\sum_{j=1}^\nu v_j s_{ij}=0,\quad i=1,\dots,n,
\end{equation*}
which means that $\mathbf S v=0$, where $\mathbf S=(s_{ij})$ is the likelihood matrix from Lemma~\ref{lem:laplace-full-rank} in the Appendix. By Lemma~\ref{lem:laplace-full-rank}, the columns of $\mathbf S$ are linearly independent, so $\mathbf S v=0$ implies $v=0$. Hence, the Hessian is negative definite, and $\ell_n$ is strictly concave on $\Delta_\nu$. Therefore, the maximizer of $\ell_n$ over $\Delta_\nu$ is unique.
\end{proof}

\section{Proofs for Section~\ref{sec:deconv}}\label{app:deconv}

\restatefrequencysplit*

\begin{proof}[Proof of Proposition~\ref{prop:freq-split}]
Taking Fourier transforms in the sense of tempered distributions, for $t\neq 0$,
\begin{equation*}
\FT{G*k_h}(t) = \frac{\FT{g}(t)\FT{k}(ht)}{it} = \frac{\FT{m}_g(t)\FT{k}(ht)(1+b^2t^2)}{it}.
\end{equation*}
By the definition of $w_{1,h,b}$ and $w_{2,h,b}$,
\begin{equation*}
\FT{k}(ht)(1+b^2t^2)=w_{1,h,b}(t)+w_{2,h,b}(t).
\end{equation*}
Therefore
\begin{equation*}
\FT{G*k_h}(t) = \frac{\FT{m}_g(t)}{it}w_{1,h,b}(t)+\FT{m}_g(t)\frac{w_{2,h,b}(t)}{it}.
\end{equation*}
Since $\FT{M}_g(t)=\FT{m}_g(t)/(it)$ and $\FT{K}_{2,h,b}(t)=w_{2,h,b}(t)/(it)$, we obtain
\begin{equation*}
\FT{G*k_h}(t) = \FT{M}_g(t) w_{1,h,b}(t)+\FT{m}_g(t) \FT{K}_{2,h,b}(t).
\end{equation*}
Inverting the Fourier transform yields \eqref{eq:FGKh-decomp}. Identity \eqref{eq:FG-diff-decomp} follows by subtracting the corresponding decomposition for $g_0$.
\end{proof}

\restatebiasbound*

\begin{proof}[Proof of Lemma~\ref{lem:bias-bound}]
Let $G$ be the distribution function of measure $g$. By Fubini's theorem and the definition of $B_{g,h}$,
\begin{align*}
\|B_{g,h}\|_1
&=\int_{\mathbb R}\left|\int_{\mathbb R} k(u)\left(G(x-hu)-G(x)\right) \mrd u\right|\mrd x \\
&\le\int_{\mathbb R}|k(u)|\left(\int_{\mathbb R}|G(x-hu)-G(x)|\mrd x\right)\mrd u.
\end{align*}
We claim that, for every distribution function $G$ and every $t\in\mathbb R$,
\begin{equation}\label{eq:cdf-shift}
\int_{\mathbb R}|G(x-t)-G(x)|\mrd x = |t|.
\end{equation}
Indeed, if $t>0$, then $G(x)\ge G(x-t)$ and
\begin{equation*}
G(x)-G(x-t)=g((x-t,x]),
\end{equation*}
Therefore,
\begin{align*}
\int_{\mathbb R}|G(x-t)-G(x)|\mrd x
&=\int_{\mathbb R}\{G(x)-G(x-t)\}\mrd x \\
&=\int_{\mathbb R}\int_{\mathbb R}\mathbf 1\{x-t<y\le x\}g(\mrd y)\mrd x \\
&=\int_{\mathbb R}\left(\int_{\mathbb R}\mathbf 1\{x-t<y\le x\}\mrd x\right) g(\mrd y) \\
&=\int_{\mathbb R} t g(\mrd y)=t.
\end{align*}
If $t<0$, applying the previous argument with $|t|$ gives
\begin{equation*}
\int_{\mathbb R}|G(x-t)-G(x)|\mrd x=\int_{\mathbb R}|G(x+|t|)-G(x)|\mrd x = |t|.
\end{equation*}
The case $t=0$ is trivial, so \eqref{eq:cdf-shift} holds for all $t\in\mathbb R$. Alternatively, one can intuitively understand this equality via optimal transport. The left-hand side is precisely the 1-Wasserstein distance between the distribution $G$ and its $t$-shifted version, which trivially evaluates to $|t|$.
Applying \eqref{eq:cdf-shift} with $t=hu$ gives
\begin{equation*}
\|B_{g,h}\|_1 \le h\int_{\mathbb R}|u| |k(u)| \mrd u = C_k h.
\end{equation*}
\end{proof}

Next, to prove Lemmas~\ref{lem:K1-bound} and \ref{lem:F2-bound} for the low-frequency and high-frequency terms, we first give the following Fourier $L^1$ estimate.

\begin{lemma}[Fourier $L^1$ estimate]\label{lem:fourier-L1}
If $f\in L^2(\mathbb R)$ is weakly differentiable with $f'\in L^2(\mathbb R)$, then
\begin{equation*}
\|\mathfrak F^{-1} (f)\|_1 \le \frac{1}{\sqrt{\pi}} \left( \|f\|_2+\|f'\|_2\right),
\end{equation*}
\end{lemma}

\begin{proof}[Proof of Lemma~\ref{lem:fourier-L1}]
Let $u:=\mathfrak F^{-1}(f)$. Then
\begin{equation*}
\|u\|_1 = \int_{|x|\le 1}|u(x)| \mrd x + \int_{|x|>1}|u(x)| \mrd x.
\end{equation*}
By Cauchy-Schwarz inequality,
\begin{equation*}
\int_{|x|\le 1}|u(x)| \mrd x\le \sqrt{2}\ \|u\|_2,
\end{equation*}
and
\begin{equation*}
\int_{|x|>1}|u(x)| \mrd x = \int_{|x|>1}|x|^{-1}|x u(x)| \mrd x \le \left(\int_{|x|>1}x^{-2} \mrd x\right)^{1/2}\|x u\|_2 =\sqrt{2}\ \|xu\|_2.
\end{equation*}
Hence
\begin{equation*}
\|u\|_1 \le \sqrt{2}\left(\|u\|_2+\|x u\|_2\right).
\end{equation*}
By Parseval–Plancherel identity~\citep[Theorem 9.13]{rudin1974real} and the identity $\mathfrak F(xu)= i f'$,
\begin{equation*}
\|u\|_2 = \frac{1}{\sqrt{2\pi}} \|f\|_2,\quad \|x u\|_2 = \frac{1}{\sqrt{2\pi}} \|f'\|_2.
\end{equation*}
Combining the two bounds proves the claim.
\end{proof}

The next two proofs follow the corresponding unit-scale arguments in \cite[Appendix~S1]{rousseau2024wasserstein}. 
To treat a general Laplace scale $b$, we introduce $\eta:=h/b$ and rewrite the multipliers in unit-scale form.

\restatelowfreq*

\begin{proof}[Proof of Lemma~\ref{lem:K1-bound}]
Let
\begin{equation*}
\eta:=h/b,\quad u:=bt,
\end{equation*}
and define the unit-scale multiplier
\begin{equation*}
w_{1,\eta}(u):=\FT{k}(\eta u)\chi(u)(1+u^2).
\end{equation*}
Then
\begin{equation*}
w_{1,h,b}(t) = \FT{k}(ht)\chi(bt)(1+b^2t^2) = \FT{k}(\eta u)\chi(u)(1+u^2) = w_{1,\eta}(u).
\end{equation*}
Now, we write
\begin{equation*}
k_{1,\eta}:=\mathfrak F^{-1}(w_{1,\eta}),
\end{equation*}
then a change of variables in the inverse Fourier transform gives
\begin{equation*}
k_{1,h,b}(x)=\frac{1}{b}k_{1,\eta}(x/b),
\end{equation*}
and therefore
\begin{equation*}
\|k_{1,h,b}\|_1=\|k_{1,\eta}\|_1.
\end{equation*}

We next use the proof of \cite[Lemma S1.1]{rousseau2024wasserstein} to show that $\|k_{1,\eta}\|_1$ stays bounded for $\eta \in (0,1/2)$.
Since $\chi$ is supported on $[-2,2]$, so is $w_{1,\eta}$. Moreover, for $|u|\le 2$,
\begin{equation*}
|w_{1,\eta}(u)|\le\|\FT{k}\mathbf{1}_{[-2\eta,2\eta]}\|_{\infty} \|\chi\|_\infty (1+u^2)
\lesssim 1,
\end{equation*}
uniformly for $\eta\in(0,1/2)$. Differentiating gives
\begin{equation*}
w_{1,\eta}'(u) = \eta \FT{k}'(\eta u)\chi(u)(1+u^2) + \FT{k}(\eta u)\chi'(u)(1+u^2) + 2u \FT{k}(\eta u)\chi(u).
\end{equation*}
Since $|u|\le 2$ on the support of $\chi$, and since $\FT{k}$ and $\FT{k}'$ are bounded on compact sets, we also have
\begin{equation*}
|w_{1,\eta}'(u)|\lesssim \mathbf{1}_{[-2,2]}(u)
\end{equation*}
uniformly for $\eta\in(0,1/2)$.

Consequently,
\begin{equation*}
\|w_{1,\eta}\|_2+\|w_{1,\eta}'\|_2\lesssim 1.
\end{equation*}
Applying Lemma~\ref{lem:fourier-L1} with $w_{1,\eta}$ yields
\begin{equation*}
\|k_{1,\eta}\|_1 = \|\mathfrak F^{-1}(w_{1,\eta})\|_1\lesssim \|w_{1,\eta}\|_2+\|w_{1,\eta}'\|_2 \lesssim 1.
\end{equation*}
Hence
\begin{equation*}
\|k_{1,h,b}\|_1=\|k_{1,\eta}\|_1\lesssim 1,
\end{equation*}
which proves \eqref{eq:K1-bound}.
\end{proof}

\restatehighfreq*

\begin{proof}[Proof of Lemma~\ref{lem:F2-bound}]
Let again
\begin{equation*}
\eta:=h/b,\quad u:=bt,
\end{equation*}
and define
\begin{equation*}
w_{2,\eta}(u):=\FT{k}(\eta u)\left(1-\chi(u)\right)(1+u^2),\quad k_{2,\eta}:=\mathfrak F^{-1}(w_{2,\eta}),
\end{equation*}
together with
\begin{equation*}
K_{2,\eta}(x):=\int_{-\infty}^x k_{2,\eta}(v) \mrd v.
\end{equation*}
Then
\begin{equation*}
w_{2,h,b}(t) = \FT{k}(ht)\left(1-\chi(bt)\right)(1+b^2t^2) = \FT{k}(\eta u)\left(1-\chi(u)\right)(1+u^2) = w_{2,\eta}(u).
\end{equation*}
As in the low-frequency case, a change of variables gives
\begin{equation*}
k_{2,h,b}(x)=\frac{1}{b}k_{2,\eta}(x/b),\quad K_{2,h,b}(x)=K_{2,\eta}(x/b),
\end{equation*}
and hence
\begin{equation*}
\|K_{2,h,b}\|_1=b \|K_{2,\eta}\|_1.
\end{equation*}
We next use the proof of \cite[Lemma S1.2]{rousseau2024wasserstein} to show that
\begin{equation*}
\|K_{2,\eta}\|_1 \lesssim |\log\eta| \eta^{-1}\quad\text{for all sufficiently small }\eta.
\end{equation*}

We decompose
\begin{equation*}
\|K_{2,\eta}\|_1 = \int_{|x|\le \eta}|K_{2,\eta}(x)| \mrd x + \int_{\eta<|x|\le 1}|K_{2,\eta}(x)| \mrd x + \int_{|x|>1}|K_{2,\eta}(x)| \mrd x =: I_1+I_2+I_3.
\end{equation*}
For the first term $I_1$,
denote
\begin{equation*}
R_k:=\sup\{|u|:\FT{k}(u)\neq 0\}<\infty.
\end{equation*}
Because $1-\chi$ vanishes on $[-1,1]$ and $\FT{k}$ is supported on $[-R_k,R_k]$, the multiplier $w_{2,\eta}$ is supported on
\begin{equation*}
D_\eta:=\{u\in\mathbb R:1<|u|\le R_k/\eta\}.
\end{equation*}
Since $w_{2,\eta}(u)/u\in L^1(\mathbb R)$, Fourier inversion gives
\begin{equation*}
K_{2,\eta}(x) = \frac{1}{2\pi}\int_{\mathbb R} e^{-iux}\frac{w_{2,\eta}(u)}{-iu} \mrd u.
\end{equation*}
Thus the first term $I_1$ can be controlled by
\begin{equation*}
I_1\le\frac{\eta}{\pi}\int_{D_\eta}\frac{|w_{2,\eta}(u)|}{|u|} \mrd u.
\end{equation*}
Since $|\FT{k}(\eta u)|\le \|\FT{k}\|_\infty$ and $|1-\chi(u)|\le 1$,
\begin{equation*}
|w_{2,\eta}(u)|\lesssim 1+u^2,\quad u\in D_\eta.
\end{equation*}
Therefore
\begin{equation*}
I_1\lesssim \eta\int_1^{R_k/\eta}\left(u+\frac{1}{u}\right)\mrd u\lesssim \eta^{-1}.
\end{equation*}

Next for $I_2$, we define
\begin{equation*}
q_\eta(u):=\frac{\mrd}{\mrd u}\left(\frac{w_{2,\eta}(u)}{-iu}\right).
\end{equation*}
For $x\neq 0$, integration by parts yields
\begin{equation*}
K_{2,\eta}(x) = \frac{1}{2\pi i x}\int_{\mathbb R} e^{-iux} q_\eta(u) \mrd u.
\end{equation*}
Hence
\begin{equation*}
I_2\le\frac{1}{2\pi}\left(\int_{\eta<|x|\le 1}\frac{1}{|x|} \mrd x\right)\|q_\eta\|_1\lesssim |\log\eta| \|q_\eta\|_1.
\end{equation*}

A direct differentiation gives
\begin{align*}
q_\eta(u)
&=\eta \FT{k}'(\eta u)\left(1-\chi(u)\right)\left(u+\frac{1}{u}\right) \\
&\quad -\FT{k}(\eta u)\chi'(u)\left(u+\frac{1}{u}\right)+\FT{k}(\eta u)\left(1-\chi(u)\right)\left(1-\frac{1}{u^2}\right),
\end{align*}
up to an irrelevant multiplicative constant of modulus one. Using that $\FT{k}$ and $\FT{k}'$ are bounded, that $\chi'$ is supported on $\{1<|u|<2\}$, and that $u\in D_\eta$ implies $1<|u|\le R_k/\eta$, we obtain
\begin{equation*}
\|q_\eta\|_1\lesssim \eta\int_1^{R_k/\eta}\left(u+\frac{1}{u}\right)\mrd u + \int_{1<|u|<2}\left(u+\frac{1}{u}\right)\mrd u + \int_1^{R_k/\eta}1 \mrd u\lesssim \eta^{-1}.
\end{equation*}
Therefore
\begin{equation*}
I_2\lesssim |\log\eta| \eta^{-1}.
\end{equation*}

For the third term $I_3$, Cauchy-Schwarz inequality gives
\begin{equation*}
I_3\le\left(\int_{|x|>1}x^{-2} \mrd x\right)^{1/2}\left(\int_{\mathbb R}|xK_{2,\eta}(x)|^2 \mrd x\right)^{1/2}\lesssim \|xK_{2,\eta}\|_2.
\end{equation*}
By Parseval–Plancherel identity~\citep[Theorem 9.13]{rudin1974real} and the representation above,
\begin{equation*}
\|xK_{2,\eta}\|_2\lesssim \|q_\eta\|_2.
\end{equation*}
Using the same decomposition of $q_\eta$,
\begin{equation*}
\|q_\eta\|_2^2\lesssim \eta^2\int_1^{R_k/\eta}\left(u+\frac{1}{u}\right)^2 \mrd u + \int_{1<|u|<2}\left(u+\frac{1}{u}\right)^2 \mrd u + \int_1^{R_k/\eta}1 \mrd u\lesssim \eta^{-1}.
\end{equation*}
Hence
\begin{equation*}
I_3\lesssim \eta^{-1/2}\le |\log\eta| \eta^{-1}
\end{equation*}
for all sufficiently small $\eta$.

Combining the bounds for $I_1$, $I_2$, and $I_3$, we obtain
\begin{equation*}
\|K_{2,\eta}\|_1\lesssim |\log\eta| \eta^{-1}.
\end{equation*}
Finally,
\begin{equation*}
\|K_{2,h,b}\|_1 = b \|K_{2,\eta}\|_1\lesssim b |\log\eta| \eta^{-1} = b^2 |\log(h/b)| h^{-1},
\end{equation*}
which proves \eqref{eq:F2-bound}.
\end{proof}

Next, to prove Lemma~\ref{lem:L1-to-W1-released}, we need to control the tails of the privatized observation density $m_g$. The next lemma gives an exponential bound on the tails.

\begin{lemma}[Laplace tails under bounded support]\label{lem:laplace-tail}
Let $g\in\mathcal P([-a,a])$, and let $M_g$ be the distribution function of $m_g=g*f_b$. Then, for all $x\ge a$,
\begin{equation}\label{eq:right-tail}
1-M_g(x)\le \frac12 e^{-(x-a)/b},
\end{equation}
and for all $x\le -a$,
\begin{equation}\label{eq:left-tail}
M_g(x)\le \frac12 e^{-(|x|-a)/b}.
\end{equation}
Consequently, for every $R\ge a$,
\begin{equation}\label{eq:tail-int}
\int_R^\infty (1-M_g(x))\mrd x \le \frac{b}{2}e^{-(R-a)/b}, \quad \int_{-\infty}^{-R} M_g(x)\mrd x \le \frac{b}{2}e^{-(R-a)/b}.
\end{equation}
\end{lemma}

\restatefromlone*

\begin{proof}[Proof of Lemma~\ref{lem:L1-to-W1-released}]
Let $T(x):=M_{g_1}(x)-M_{g_2}(x)$. For any $R\ge a$,
\begin{equation*}
\int_{\mathbb R}|T(x)|\mrd x = \int_{|x|\le R}|T(x)| \mrd x + \int_{|x|>R}|T(x)| \mrd x.
\end{equation*}
On $[-R,R]$, using $\sup_x |T(x)|\le \mathrm{TV}(m_{g_1},m_{g_2})=\|m_{g_1}-m_{g_2}\|_1/2=d/2$, we get
\begin{equation*}
\int_{|x|\le R}|T(x)|\mrd x \le Rd.
\end{equation*}
On the tails, use
\begin{equation*}
|T(x)|\le (1-M_{g_1}(x))+(1-M_{g_2}(x)) \quad \text{for } x>R,
\end{equation*}
and
\begin{equation*}
|T(x)|\le M_{g_1}(x)+M_{g_2}(x)\quad \text{for } x < -R,
\end{equation*}
then apply Lemma~\ref{lem:laplace-tail} to both $M_{g_1}$ and $M_{g_2}$. This gives
\begin{equation*}
\int_{|x|>R}|T(x)|\mrd x \le 2b e^{-(R-a)/b},
\end{equation*}
which proves \eqref{eq:W1-trunc}. Choosing $R=a+b\log(1/d)$ yields \eqref{eq:L1-to-W1-released-final}.
\end{proof}

\section{Proofs for Section~\ref{sec:rate}}\label{app:rate}

We begin with the auxiliary results used in the proof of Theorem~\ref{thm:direct-rate}. 
The proof follows the convex class likelihood argument of \cite{van1996rates}, with the constants written for the Laplace kernel with scale $b$.
Here, Lemma~\ref{lem:envelope-bound} shows that under the conditions of Theorem~\ref{thm:direct-rate}, we have a uniform envelope bound for the Laplace kernels. 
After that, Proposition~\ref{prop:entropy-bound} uses this envelope control together with a convex-hull entropy theorem to bound the metric entropy of the mixture class. 
Finally, these bounds are combined with the convex-class likelihood theory of \cite{van1996rates} to derive the Hellinger rate for the privatized observation density. 
We start with the following Lemma.

\begin{restatable}[Envelope bound]{lemma}{restateenvelopebound}\label{lem:envelope-bound}
Under the conditions of Theorem~\ref{thm:direct-rate}, for every $x\in\mathbb R$,
\begin{equation}\label{eq:envelope-bound}
\sup_{|\theta|\le a}\frac{f_b(x-\theta)}{m_{g_0}(x)}\le e^{2a/b}.
\end{equation}
\end{restatable}

\begin{proof}[Proof of Lemma~\ref{lem:envelope-bound}]
Let $\pi_x:=\Pi_{[-a,a]}(x)$ be the Euclidean projection of $x$ onto $[-a,a]$. Since $\inf_{|\theta|\le a}|x-\theta| = |x-\pi_x|$, the Laplace density satisfies
\begin{equation}\label{eq:max-kernel-projection}
\sup_{|\theta|\le a} f_b(x-\theta)=f_b(x-\pi_x).
\end{equation}
Moreover, for every $\theta\in[-a,a]$,
\begin{equation*}
f_b(x-\theta) = \frac{1}{2b}e^{-|x-\theta|/b} = \frac{1}{2b}e^{-|x-\pi_x|/b}e^{-|\theta-\pi_x|/b} = f_b(x-\pi_x)e^{-|\theta-\pi_x|/b}.
\end{equation*}
Integrating against $g_0$ gives
\begin{align}
m_{g_0}(x)
&= \int_{-a}^a f_b(x-\theta) g_0(\mrd \theta) \notag\\
&= f_b(x-\pi_x)\int_{-a}^a e^{-|\theta-\pi_x|/b}g_0(\mrd\theta) \notag\\
&\ge e^{-2a/b} f_b(x-\pi_x) \notag\\
&= e^{-2a/b}\sup_{|\theta|\le a}f_b(x-\theta).
\label{eq:p0-factorization}
\end{align}
which proves \eqref{eq:envelope-bound}.
\end{proof}

The following convex-hull entropy bound is quoted from \cite[Theorem~1.1]{van1996rates}.

\begin{lemma}[Convex-hull entropy bound~\cite{van1996rates}]\label{lem:convex-hull-entropy}
Let $(E,\|\cdot\|)$ be a normed linear space, and let $\mathcal F\subset E$. 
Let
$V(\gamma,\mathcal F,\|\cdot\|)$ denote the $\gamma$-covering number of $\mathcal F$, that is, the smallest integer $V$ such that there exist $\xi_1,\ldots,\xi_V\in E$ with
\begin{equation*}
\mathcal F\subseteq \bigcup_{j=1}^V \{\xi\in E:\|\xi-\xi_j\|\le \gamma\}.
\end{equation*}
Assume that for some $A>0$,
\begin{equation*}
V(\gamma,\mathcal F,\|\cdot\|)\le \frac{A}{\gamma} \quad \text{for all } \gamma>0.
\end{equation*}
Let $\mathrm{conv}(\mathcal F)$ denote the convex hull of $\mathcal F$.
Then there exists a universal constant $C>0$ such that
\begin{equation*}
\log V(\gamma,\mathrm{conv}(\mathcal F),\|\cdot\|) \le C\left(\frac{A}{\gamma}\right)^{2/3} \quad \text{for all } \gamma>0.
\end{equation*}
\end{lemma}

We now bound the metric entropy of the kernel class. Let
\begin{equation*}
\mathcal K_{a,b}:=\left\{\frac{f_b(\cdot-\theta)}{m_{g_0}(\cdot)}:\ \theta\in[-a,a]\right\}.
\end{equation*}
For $g\in\mathcal P([-a,a])$, the ratio $m_g/m_{g_0}$ belongs to the closed convex hull of $\mathcal K_{a,b}$.
The next proposition applies the preceding convex hull bound to this normalized kernel class.
This proposition is an application of the convex hull entropy bound of~\cite{van1996rates}, with constants specialized to the normalized Laplace kernel class considered here; see also~\cite{scricciolo2018bayes} for related entropy arguments for Laplace mixtures.

\begin{restatable}[Entropy bound for the kernel class]{proposition}{restateentropyboundkernel}\label{prop:entropy-bound}
Under the conditions of Theorem~\ref{thm:direct-rate}, there exists a constant $C>0$, independent of $a$, $b$, and $g_0$, such that for every probability measure $Q$ on $\mathbb R$ and every $\gamma>0$,
\begin{equation}\label{eq:entropy-bound}
\log V \left(\gamma, \mathrm{conv}(\mathcal K_{a,b}), \|\cdot\|_{2,Q} \right)
\le C\left(\frac{(1+a/b)e^{2a/b}}{\gamma}\right)^{2/3}.
\end{equation}
\end{restatable}

\begin{proof}[Proof of Proposition~\ref{prop:entropy-bound}]
Fix $\theta_1,\theta_2\in[-a,a]$. By the mean value theorem and the identity
\begin{equation*}
|f_b'(t)|=\frac{1}{b}f_b(t) \quad \forall t\neq 0,
\end{equation*}
we have, pointwise in $x$,
\begin{equation*}
|f_b(x-\theta_1)-f_b(x-\theta_2)|
\le \frac{|\theta_1-\theta_2|}{b} \sup_{\theta\in[\theta_1,\theta_2]} f_b(x-\theta) \le \frac{|\theta_1-\theta_2|}{b}\sup_{|\theta|\le a} f_b(x-\theta).
\end{equation*}
Dividing by $m_{g_0}(x)$ and applying Lemma~\ref{lem:envelope-bound}, we obtain
\begin{equation*}
\left| \frac{f_b(x-\theta_1)}{m_{g_0}(x)}-\frac{f_b(x-\theta_2)}{m_{g_0}(x)} \right| \le \frac{|\theta_1-\theta_2|}{be^{-2a/b}}.
\end{equation*}
Therefore, for every probability measure $Q$ on $\mathbb R$,
\begin{equation*}
\left\| \frac{f_b(\cdot-\theta_1)}{m_{g_0}} - \frac{f_b(\cdot-\theta_2)}{m_{g_0}} \right\|_{2,Q} \le \frac{|\theta_1-\theta_2|}{be^{-2a/b}}.
\end{equation*}

Hence the class
\begin{equation*}
\mathcal K_{a,b} = \left\{ \frac{f_b(\cdot-\theta)}{m_{g_0}(\cdot)}:\ \theta\in[-a,a] \right\}
\end{equation*}
is $(be^{-2a/b})^{-1}$-Lipschitz in the parameter $\theta\in[-a,a]$ under the metric $\|\cdot\|_{2,Q}$. 
A grid on $[-a,a]$ with mesh size $b\gamma e^{-2a/b}$ therefore yields
\begin{equation*}
V\left(\gamma,\mathcal K_{a,b},\|\cdot\|_{2,Q}\right) \le 1+\frac{2a}{b\gamma e^{-2a/b}} \lesssim
\frac{(1+a/b) e^{2a/b}}{\gamma}.
\end{equation*}
Now apply Lemma~\ref{lem:convex-hull-entropy} with
\begin{equation*}
A_{a,b} := (1+a/b) e^{2a/b},
\end{equation*}
to conclude that
\begin{equation*}
\log V \left( \gamma, \mathrm{conv}(\mathcal K_{a,b}), \|\cdot\|_{2,Q} \right) \le C\left(\frac{(1+a/b) e^{2a/b}}{\gamma}\right)^{2/3}.
\end{equation*}
This proves \eqref{eq:entropy-bound}.
\end{proof}

We can now pass from entropy control of the mixture class to a likelihood rate for the privatized observation density. 

\restatehellingerrate*

\begin{proof}[Proof of Theorem~\ref{thm:direct-rate}]
Let
\begin{equation*}
\mathcal F_{a,b}:=\{m_g: m_g=g*f_b, g\in\mathcal P([-a,a])\}.
\end{equation*}
This is a convex class of densities, and $m_{\widehat g_n}$ is its maximum likelihood estimator. 
For $p\in\mathcal F_{a,b}$, define
\begin{equation*}
\psi_p:=\frac{2p}{p+m_{g_0}}-1=\frac{p-m_{g_0}}{p+m_{g_0}}, \quad \mathcal \Psi_{a,b}:=\{\psi_p:p\in\mathcal F_{a,b}\}.
\end{equation*}

We apply Theorem~2.2 of \cite{van1996rates} with $\mathcal F=\mathcal F_{a,b}$ and $\sigma_n\equiv 0$. Let
\begin{equation*}
A_{a,b}:=(1+a/b)e^{2a/b}.
\end{equation*}
Fix any probability measure $Q$ on $\mathbb R$. For $p_1,p_2\in\mathcal F_{a,b}$,
\begin{equation*}
|\psi_{p_1}-\psi_{p_2}| = \left| \frac{2p_1}{p_1+m_{g_0}}-\frac{2p_2}{p_2+m_{g_0}} \right| = \frac{2m_{g_0}|p_1-p_2|}{(p_1+m_{g_0})(p_2+m_{g_0})} \le \frac{2|p_1-p_2|}{m_{g_0}}.
\end{equation*}
Define $\mathcal F_{a,b}/m_{g_0}:=\left\{m_g/m_{g_0} : m_{g}\in \mathcal F_{a,b} \right\}$.
Therefore,
\begin{equation*}
V\left(\gamma,\mathcal \Psi_{a,b},\|\cdot\|_{2,Q}\right) \le V\left(\gamma/2,\mathcal F_{a,b}/m_{g_0},\|\cdot\|_{2,Q}\right).
\end{equation*} 
Note that for every $g\in\mathcal P([-a,a])$,
\begin{equation*}
\frac{m_g(x)}{m_{g_0}(x)}=\int_{[-a,a]} \frac{f_b(x-\theta)}{m_{g_0}(x)} g(\mrd\theta),
\end{equation*}
so $\mathcal F_{a,b}/m_{g_0}$ is the closed convex hull of $\mathcal K_{a,b}$.
Applying Proposition~\ref{prop:entropy-bound} yields
\begin{equation*}
\log V \left(\gamma,\mathcal \Psi_{a,b},\|\cdot\|_{2,Q}\right) \lesssim \left(\frac{A_{a,b}}{\gamma}\right)^{2/3}
\end{equation*}
uniformly over all probability measures $Q$ on $\mathbb R$, and in particular for empirical measure $Q=\frac1n\sum_{i=1}^n\delta_{X_i}$, where we set $A_{a,b} = (1+a/b)e^{2a/b}$.

Hence Theorem~2.2 of \cite{van1996rates} applies with entropy exponent $2/3$ and scale parameter $\rho_n\asymp A_{a,b}$. The resulting rate is
\begin{equation*}
n^{-1/(2+2/3)}\rho_n^{(2/3)/(2+2/3)} \asymp n^{-3/8}A_{a,b}^{1/4}.
\end{equation*}
It follows that the Hellinger distance between $m_{\widehat g_n}$ and $m_{g_0}$ satisfied
\begin{equation*}
H(m_{\widehat g_n},m_{g_0})=\mathcal O_p \left(n^{-3/8}A_{a,b}^{1/4}\right),
\end{equation*}
which gives \eqref{eq:direct-rate}.
\end{proof}

\restategrowthcondition*

\begin{proof}[Proof of Corollary~\ref{cor:b-growth}]
Let
\begin{equation*}
r_n:=(1+a/b_n)^{1/4}e^{a/(2b_n)}n^{-3/8}.
\end{equation*}
Since $b_n$ is nondecreasing,
\begin{equation*}
(1+a/b_n)^{1/4}e^{a/(2b_n)}\le(1+a/b_1)^{1/4}e^{a/(2b_1)},
\end{equation*}
so $r_n\lesssim n^{-3/8}$. Also $(1+a/b_n)^{1/4}e^{a/(2b_n)}\ge 1$, hence $r_n\asymp n^{-3/8}$ and $\log(1/r_n)\asymp \log n$.

Applying Corollary~\ref{cor:w1-rate}, we obtain
\begin{equation*}
W_1(\widehat g_n,g_0) = \mathcal O_p\left( a r_n+b_n\sqrt{r_n\log(1/r_n)} \right)
= \mathcal O_p \left( a n^{-3/8} + b_n n^{-3/16}\sqrt{\log n} \right),
\end{equation*}
which proves \eqref{eq:w1-rate-bn}. The consistency statement follows from \eqref{eq:b-sufficient}.
\end{proof}

\begin{lemma}[Bretagnolle-Huber inequality~\cite{bretagnolle2006estimation,lattimore2020bandit}]\label{lem:BH}
For any two probability measures $\mathbb P,\mathbb Q$ and any measurable test $\phi$,
\begin{equation*}
\mathbb P(\phi=0)+\mathbb Q(\phi=1) \ge
\frac12\exp \left(-\mathrm{KL}(\mathbb P || \mathbb Q)\right).
\end{equation*}
\end{lemma}

\begin{proof}[Proof of Lemma~\ref{lem:BH}]
See \cite[Theorem 14.2]{lattimore2020bandit}.
\end{proof}

\restateimpossibility*

\begin{proof}[Proof of Theorem~\ref{thm:impossibility-sqrtn}]
Consider two atoms $g_+ := \delta_a$ and $g_- := \delta_{-a}$ which both belong to $\mathcal P([-a,a])$. Clearly,
\begin{equation}\label{eq:w1-two-point-distance}
W_1(g_+,g_-)=2a.
\end{equation}
Note that $p_{g_+, b_n}$ is the Laplace density centered at $a$ with scale $b_n$, and $p_{g_-, b_n}$ is the Laplace density centered at $-a$ with scale $b_n$. Moreover, their Kullback-Leibler divergence is given by
\begin{align}
\mathrm{KL}(p_{g_+, b_n} || p_{g_-, b_n})
& = \int_{\mathbb R} p_{g_+, b_n}(x) \log \frac{p_{g_+, b_n}(x)}{p_{g_-, b_n}(x)} \mathrm{d} x \notag \\
& = \int_{\mathbb R} f_{b_n}(x-a) \log \frac{\frac{1}{2b_n} \exp(-|x-a|/b_n)}{\frac{1}{2b_n} \exp(-|x+a|/b_n)} \mathrm{d} x \notag \\
& = \int_{\mathbb R} f_{b_n}(x-a) \left( \frac{|x+a| - |x-a|}{b_n} \right) \mathrm{d} x \notag \\
& = \frac{1}{b_n} \left( \mathbb{E}_{X \sim \text{Lap}(a, b_n)} [|X+a|] - \mathbb{E}_{X \sim \text{Lap}(a, b_n)} [|X-a|] \right) \notag \\
& = \frac{1}{b_n} \left( \left[ 2a + b_n e^{-2a/b_n} \right] - b_n \right) \notag \\
& = \frac{2a}{b_n} + e^{-2a/b_n} - 1.
\end{align}

Let $\mathbb P_{+,n}$ and $\mathbb P_{-,n}$ denote the joint distribution of
$(X_{1,n},\dots,X_{n,n})$ under $g_+$ and $g_-$, respectively. Then we have
\begin{equation*}
\mathrm{KL}(\mathbb P_{+,n} || \mathbb P_{-,n}) = n\left(\frac{2a}{b_n}+e^{-2a/b_n} -1\right).
\end{equation*}
Now use the elementary inequality $e^{-x}\le 1-x+{x^2}/{2}$ and combining with $b_n\ge c\sqrt n$,
\begin{equation}\label{eq:kl-upperbound}
\mathrm{KL}(\mathbb P_{+,n} || \mathbb P_{-,n}) \le n\cdot \frac12\left(\frac{2a}{b_n}\right)^2 = \frac{2a^2n}{b_n^2} \le \frac{2a^2}{c^2}
\end{equation}
holds for all sufficiently large $n$.

Now let $g_n^\star$ be any estimator based on the privatized sample, and define the test
\begin{equation*}
\phi_n:= \mathbf 1 \left\{W_1(g_n^\star,g_+)<a\right\}.
\end{equation*}
By \eqref{eq:w1-two-point-distance} and the triangle inequality, the two balls
\begin{equation*}
\left\{g:W_1(g,g_+)<a\right\} \quad\text{and}\quad \left\{g:W_1(g,g_-)<a\right\}
\end{equation*}
are disjoint. Therefore we have $\{W_1(g_n^\star,g_+)<a\}\subseteq\{W_1(g_n^\star,g_-)\ge a\}$, and hence
\begin{align}
&\quad\ \mathbb P_{+,n}\left(W_1(g_n^\star,g_+)\ge a\right)+\mathbb P_{-,n}\left(W_1(g_n^\star,g_-)\ge a\right) \notag\\
& \ge \mathbb P_{+,n}\left(W_1(g_n^\star,g_+)\ge a\right)+\mathbb P_{-,n}\left(W_1(g_n^\star,g_+)<a\right)\notag\\
& = \mathbb P_{+,n}(\phi_n=0)+\mathbb P_{-,n}(\phi_n=1). \notag
\end{align}
Applying Lemma~\ref{lem:BH} gives
\begin{equation*}
\mathbb P_{+,n}(\phi_n=0)+\mathbb P_{-,n}(\phi_n=1)\ge \frac12\exp \left(-\mathrm{KL}(\mathbb P_{+,n} || \mathbb P_{-,n})\right) \ge \frac12\exp \left(-\frac{2a^2}{c^2}\right),
\end{equation*}
therefore,
\begin{align*}
&\quad\ \sup_{g\in\mathcal P([-a,a])}
\mathbb P_{g}^{n} \left(W_1(g_n^\star,g)\ge a\right) \\
&\ge \max \left\{ \mathbb P_{+,n} \left(W_1(g_n^\star,g_+)\ge a\right), \mathbb P_{-,n} \left(W_1(g_n^\star,g_-)\ge a\right) \right\} \\
&\ge \frac12 \left(\mathbb P_{+,n}\left(W_1(g_n^\star,g_+)\ge a\right)+\mathbb P_{-,n}\left(W_1(g_n^\star,g_-)\ge a\right)\right) \\
&\ge \frac14\exp \left(-\frac{2a^2}{c^2}\right).
\end{align*}
Taking $\liminf_{n\to\infty}$ proves \eqref{eq:minimax-impossibility}.

Finally, if there existed an estimator sequence $g_n^\star$ such that
\begin{equation*}
W_1(g_n^\star,g)\to0 \quad\text{in }\mathbb P_{g}^{n}\text{-probability for every }g\in\mathcal P([-a,a]),
\end{equation*}
then in particular
\begin{equation*}
\mathbb P_{+,n} \left(W_1(g_n^\star,g_+)\ge a\right)\to 0 \quad\text{and}\quad \mathbb P_{-,n} \left(W_1(g_n^\star,g_-)\ge a\right)\to 0,
\end{equation*}
which contradicts the positive lower bound. 
\end{proof}
\end{appendix}

\end{document}